\documentclass{amsart}
\usepackage{amsfonts,amscd, amssymb}

\newtheorem{theorem}{Theorem}[section]
\newtheorem{lemma}[theorem]{Lemma}
\newtheorem{corollary}[theorem]{Corollary}
\newtheorem{proposition}[theorem]{Proposition}
\theoremstyle{remark}

\theoremstyle{definition}

\numberwithin{equation}{section}
\makeatother

\DeclareMathOperator{\Cdb}{{\mathbb C}}
\DeclareMathOperator{\Rdb}{{\mathbb R}}
\DeclareMathOperator{\Ddb}{{\mathbb D}}
\DeclareMathOperator{\Tdb}{{\mathbb T}}
\DeclareMathOperator{\Ndb}{{\mathbb N}}

\begin{document}

\title[Order theory and operator algebras]{Order theory and interpolation in operator algebras} 

\date{\today}
\thanks{The first author was supported by NSF grant  DMS 1201506.   
The second author is grateful for  support from UK research council
grant  EP/K019546/1.  Many of the more recent results here were presented at COSy 2014 at the Fields Institute.}

\author{David P. Blecher}
\address{Department of Mathematics, University of Houston, Houston, TX
77204-3008}
\email[David P. Blecher]{dblecher@math.uh.edu}

\author{Charles John Read}
\address{Department of Pure Mathematics,
University of Leeds,
Leeds LS2 9JT,
England}
 \email[Charles John Read]{read@maths.leeds.ac.uk}

\begin{abstract}    In earlier papers we have introduced and studied a new 
notion of positivity in operator algebras,  with an eye to
extending certain $C^*$-algebraic results and theories to more 
general algebras.   
Here we continue to  develop
this positivity and its associated ordering, proving many 
foundational facts.  We also give many applications, for example
 to noncommutative topology, noncommutative peak sets,
lifting problems,
peak interpolation, approximate identities, and to order relations between an operator 
algebra and the $C^*$-algebra it generates.   In much of this it is not necessary that the algebra have an approximate identity.
Many of our results apply immediately to function algebras, but we will not 
take the time to point these out, although most of these applications seem new.    
  \end{abstract}

\maketitle

\section{Introduction} 

An {\em operator algebra} is a closed subalgebra of $B(H)$, for a
Hilbert space $H$.  In much of the paper our operator algebras have contractive approximate
identities (cai's), and we call auch algebras {\em approximately unital}. 
 
In earlier papers \cite{BRI,BRII,Read} we introduced and studied a new notion of positivity in operator algebras.
We have shown elsewhere that the `completely  positive' maps on C*-algebras or operator systems in our new sense
are precisely the completely positive maps in the usual sense; however the new 
notion of positivity allows the development of useful order theory for more general spaces and algebras. 
Our main goals are to 
extend certain useful $C^*$-algebraic results and theories to more 
general algebras; and also to develop `noncommutative function theory'
in the sense of generalizing certain parts of 
the classical theory of function spaces and algebras \cite{Gam}.  Simultaneously we are 
developing applications (see also e.g.\ \cite{BNI,BNII} with Matthew Neal, and \cite{Bnpi}).  
With the same goals in mind, in the present paper, we continue the development
of foundational aspects of this positivity, and 
of the associated ordering for operator algebras.  We also give many applications, for example
 to noncommutative topology, noncommutative peak sets,
lifting problems,
peak interpolation, approximate identities, and to order relations between an operator 
algebra and the $C^*$-algebra it generates. 

Before proceeding further, 
we make an editorial/historical note:  approximately half  of 
the present paper was formerly part of a preprint \cite{BRIII}.  The latter
 has been split into several papers, each of which has taken on
a life of its own, e.g.\ the present paper which focuses on 
order in operator algebras,
and  \cite{BOZ} which covers the more 
general setting of Banach algebras.  The reader is encouraged to browse the latter paper for complementary theory;
we will not prove results here that may be found in \cite{BOZ} except if there
is a much simpler proof in the operator algebra setting.   

As in the aforementioned papers, a central role is played by the set ${\mathfrak F}_A = \{ a \in A : 
\Vert 1 - a \Vert \leq 1 \}$ (here $1$ is the identity of the unitization $A^1$ if $A$ is nonunital).   We will be 
interested in four `cones' or  notions of `positivity' in $A$, and the relations between them.
The biggest of these is the set of {\em accretive} operators
$${\mathfrak r}_A = \{ a \in A : {\rm Re}(x)  \, = \, a + a^* \geq 0 \},$$
namely the elements of $A$ whose numerical range in $A^1$ is contained in the closed right half plane.
This has as a dense subcone 
$${\mathfrak c}_A = 
\Rdb_+ \, {\mathfrak F}_A,$$
(see e.g.\ \cite[Theorem 3.3]{BRII}).  In turn the latter cone 
contains as a dense subcone (see Lemma \ref{rootf})  the cone of sectorial operators of angle $\rho < \frac{\pi}{2}$, which we use  less frequently.
By 
sectorial angle $\rho$ we mean that the numerical
range is contained in the sector $S_\rho$ consisting of numbers $re^{i \theta}$ with
argument $\theta$ such that $|\theta| < \rho$ (c.f.\  e.g.\ \cite{Haase,NF}).     The fourth notion, 
`near positivity', is more subtle.   If in the statement of a result an element of $A$ is described as `nearly positive', this means that if $\epsilon > 0$ is given one 
can choose $x$ in the statement to be in the previous three cones, but also 
sectorial with angle $\rho$ so small that  $x$ is within distance $\epsilon$ of an actual positive operator.
Note that if an operator $x$ is sectorial with angle $\rho$ so small that $\Vert x \Vert \, \sin \rho < \epsilon$ for example, then Re$(x) \geq 0$ and 
$$\Vert x - {\rm Re} \, (x) \Vert = \Vert {\rm Im} \, (x) \Vert = \sup \{ | {\rm Im} \,  \langle x \zeta , \zeta \rangle | : \zeta \in {\rm Ball}(H) \} < \epsilon,$$
 so that $x$ is within distance $\epsilon$ of the positive operator Re$(x)$.    Such nearly positive operators  usually arise because 
${\mathfrak r}_A$ is closed under taking (principal) roots, and the $n$th root of an accretive operator is sectorial with angle as small as desired for $n$ large enough.     We will also usually require our
nearly positive operators to be in $\frac{1}{2} {\mathfrak F}_A$ too.

Elements of these `cones', and their roots, play the role in many situations of positive elements
in a $C^*$-algebra.    There are some remarkable relationships between operator algebras and the 
classical theory of ordered linear spaces (due to Krein, Ando, Alfsen, and many others).
We mention some examples of this (and see \cite{BOZ}, particularly Section 6 there, for
 more): 
In the language of ordered Banach spaces, an operator algebra is approximately unital iff  ${\mathfrak r}_A$ 
and ${\mathfrak c}_A$ are  {\em generating} cones (this is sometimes 
called {\em positively generating} or {\em directed} or {\em co-normal}).  That is, iff $A = {\mathfrak r}_A - {\mathfrak r}_A$,
for example.   Read's theorem states that any approximately unital operator algebra
has a cai in $\frac{1}{2} {\mathfrak F}_A$ (see \cite{Read}, although there are now several  much shorter proofs
\cite{Bnpi,BOZ}), and indeed, by taking roots, nearly positive.   We will
show that $A$ is {\em cofinal} in any $C^*$-algebra $B$ which it generates.
Indeed given any $b \in B_+$ and $\epsilon > 0$
there exists nearly positive $a \in A$
with $b \preccurlyeq a \preccurlyeq (\Vert b \Vert + \epsilon) 1$ in the ordering induced by the cone above.
We will also investigate the relationship between such results
 and `noncommutative peak interpolation'.

Turning to the layout of our paper, in Section
2 we study general properties of these cones and the related ordering.  This is
a collection of results on positivity, some of which are used elsewhere  in this 
paper, or in other papers, and some of which are of independent interest.   
In particular we prove several surprizing order theoretic
properties, some of which are new relations between an operator algebra
and the $C^*$-algebra it generates.  Many of these order theoretic properties turn out to be equivalent 
to the existence of a cai. 
The short Section 3 studies  `strictly positive' elements, a topic that is quite important for $C^*$-algebras.  
The lengthy Section 4 concerns applications to noncommutative topology, noncommutative peak sets,
lifting problems,
 and peak interpolation.
First we 
present versions of some of our previous Urysohn lemmas and peak interpolation 
results for operator algebras (see e.g.\ \cite{BRI, BNII}), but 
now insisting that the `interpolating element'
is `nearly positive'  in the sense defined above (and also in
$\frac{1}{2} {\mathfrak F}_A$).  This also 
solves the problems raised at the end of \cite{BNII}.   We also
prove a Tietze extension theorem for operator algebras, and a strict form of the 
Urysohn lemma for operator algebras, generalizing the usual strict form of the 
Urysohn lemma from topology, and also generalizing Pedersen's strict noncommutative Urysohn lemma
from \cite{PedS}.  
  See \cite{CGK} for a recent paper 
containing  `Urysohn lemmas' for function algebras; our Urysohn lemma applied to the algebras considered there is more general  (see
the discussion after Theorem \ref{urys}).  
Indeed many results in our paper apply immediately to function algebras (uniform algebras),
that is to uniformly closed subalgebras of $C(K)$, since these
are special cases of operator algebras.  We will not take the time
to point these out, although most of these applications seem new. 

We now turn to notation and some background facts (for more details the reader should consult our previous papers
in this series 
and \cite{BLM}).
 In this paper $H$ will always be a Hilbert space, usually the Hilbert space
on which our operator algebra is acting, or is completely isometrically represented.
We recall that by a theorem due to  Ralf Meyer,
every operator algebra $A$ has a  unitization $A^1$ 
which is unique up
to completely isometric homomorphism (see
 \cite[Section 2.1]{BLM}). Below $1$ always refers to
the identity of $A^1$ if $A$ has no identity.   We almost always
set $A^1 = A$ if $A$ already has an identity.  We write 
oa$(x)$  for the operator algebra generated by $x$ in $A$,
the smallest closed subalgebra containing $x$.   We will often use 
$C^*$-algebras generated by an operator algebra $A$  (or containing $A$ completely 
isometrically as a subalgebra).   For example,
the disk algebra $A(\Ddb)$ generates $C(\Tdb), C(\bar{\Ddb})$, and the Toeplitz $C^*$-algebra (here $\Tdb$ and $\Ddb$
represent the circle and open unit disk respectively).  However we want anything we say about an operator algebra $A$ 
to be independent of which particular generated $C^*$-algebra was used.  
A {\em state} of an approximately unital
operator algebra $A$ is a functional with $\Vert \varphi \Vert = \lim_t \, \varphi(e_t) = 1$ 
for some (or any) cai $(e_t)$ for $A$.  These extend to states of $A^1$.  They also 
extend to a state on any $C^*$-algebra $B$ generated by $A$, and conversely
any state on $B$ restricts to a state of $A$.  See \cite[Section 2.1]{BLM}
for details.    If $A$ is not approximately unital then we define a state on $A$  
to be a norm $1$ functional that extends to a state
on $A^1$.
We  write $S(A)$ for the collection of such states; this is the  {\em state space} of $A$.   These 
extend further by the Hahn-Banach theorem to a state
on any $C^*$-algebra generated by $A^1$, and therefore restrict 
to a positive functional on any $C^*$-algebra $B$ generated by $A$. 
The latter restriction is actually a state, since it has norm $1$
(even on  $A$).  
Conversely, every state on $B$ 
extends to a state on $B^1$, and this restricts to a state on $A^1$.
From these considerations it is easy to see that states on an operator algebra $A$ may equivalently 
be defined to be norm $1$ functionals that extend to a state
on any $C^*$-algebra $B$ generated by $A$.

For  us a {\em projection}
is always an orthogonal projection, and an {\em idempotent} merely
satisfies $x^2 = x$. If $X, Y$ are sets, then $XY$ denotes the
closure of the span of products of the form $xy$ for $x \in X, y \in
Y$.   We write $X_+$ for the positive operators (in the usual sense) that happen to
belong to $X$. We write $M_n(X)$ for the space of $n \times n$ matrices 
over $X$, and of course $M_n = M_n(\Cdb)$.  
  The
second dual $A^{**}$ is also an operator algebra with its (unique)
Arens product, this is also the product inherited from the von Neumann
algebra $B^{**}$ if
$A$ is a subalgebra of a $C^*$-algebra $B$. 
Note that
$A$ has a cai iff $A^{**}$ has an identity $1_{A^{**}}$ of norm $1$,
and then $A^1$ is sometimes identified with $A + \Cdb 1_{A^{**}}$.

For an operator algebra, not necessarily approximately  unital,
we recall that $\frac{1}{2} {\mathfrak F}_A = \{ a \in A : \Vert 1 - 2 a \Vert \leq 1 \}$.   Here $1$ is
the identity of the unitization $A^1$ if $A$ is nonunital.  
As we said, $A^1$ is uniquely defined, and can be viewed 
as $A + \Cdb I_H$ if $A$ is completely isometrically represented 
as a subalgebra of $B(H)$.  Hence so is
$A^1 + (A^1)^*$ uniquely defined, by e.g.\ 1.3.7 in \cite{BLM}.  
We define $A + A^*$ to be the obvious subspace
of $A^1 + (A^1)^*$.  This is well defined independently 
of the particular Hilbert space $H$ on which $A$
is represented, as shown at the start of Section 3 in \cite{BRII}. 
 Thus a statement such as
$a + b^* \geq 0$ makes sense whenever $a, b \in A$, and is independent of the particular $H$ on which $A$
is represented.  
This gives another way of seeing that  the set ${\mathfrak r}_A = \{ a \in A : a + a^* \geq 0 \}$ is
independent of the particular representation too.  

Note that $x \in {\mathfrak c}_A = \Rdb_+ {\mathfrak F}_A$  iff there is a positive
constant $C$ with $x^* x \leq C(x+x^*)$.  

We recall that an {\em r-ideal} is a right ideal with a left cai, and an {\em $\ell$-ideal} is a left ideal with a right cai.
We say that an operator algebra $D$ with cai, which is a subalgebra of
another operator algebra $A$, is a HSA (hereditary subalgebra)
in $A$, if $DAD \subset D$.    See
\cite{BHN} for the basic theory of HSA's.
HSA's in $A$ are in an order preserving,
bijective correspondence with the r-ideals in $A$, and
 with the $\ell$-ideals in $A$.   Because of this symmetry
we will usually restrict our results to the r-ideal case; the $\ell$-ideal case will be analogous.
There is also a bijective correspondence with the
{\em open projections} $p \in A^{**}$, by which we mean that there
is a net $x_t \in A$ with $x_t = p x_t  \to p$ weak*,
or equivalently with $x_t = p x_t p \to p$ weak* (see  \cite[Theorem 2.4]{BHN}).  These are
also the open projections $p$ in the sense of Akemann \cite{Ake2} in $B^{**}$, where $B$ is a $C^*$-algebra containing $A$, such that
$p \in A^{\perp \perp}$.   If $A$ is approximately unital then the complement $p^\perp = 1_{A^{**}} - p$
 of an open projection for $A$
 is called a {\em closed projection} for $A$.   A closed projection $q$ for which there exists an $a \in {\rm Ball}(A)$ with 
$aq = qa = q$ is called {\em compact}.  This is equivalent to 
$A$ being a closed projection with respect to 
$A^1$, if $A$ is approximately unital.  See \cite{BNII, BRII} for the theory of compact projections in operator algebras.  

If $x \in {\mathfrak r}_A$ then it is shown in \cite[Section 3]{BRII}  that the operator algebra oa$(x)$ generated by $x$ in $A$ has a cai, which can be taken to be a normalization of $(x^{\frac{1}{n}})$, and the weak* limit of $(x^{\frac{1}{n}})$  is the support projection $s(x)$ for $x$.  This is an open projection, and in 
a separable operator algebra these are the only open projections.  In a unital operator algebra the complement of an open projection (different from $1$)
 is
a peak projection, thus in separable unital operator algebra the peak projections are exactly the closed projections.  There are many equivalent definitions of peak projections (see e.g.\ \cite{Hay,BHN,BNII,BRII}).  
For any operator algebra $A$ we recall that if  $x$ is a norm $1$ element
of $\frac{1}{2} \, {\mathfrak F}_A$ then the {\em peak projection} associated with 
$x$  is   $u(x) = \lim_n \, x^n$.  This is the weak* limit, which always exists and is
nonzero  \cite[Corollary 3.3]{BNII}.  For other contractions $x$ this weak* limit may not exist or may be zero,  
but if  this weak* limit does exist and is nonzero then it is a peak projection.  We have $u(x^{\frac{1}{n}}) = u(x)$, for 
$x \in \frac{1}{2} \, {\mathfrak F}_A$ (see \cite[Corollary 3.3]{BNII}).  Compact projections
in approximately unital algebras are precisely the infima (or decreasing weak* limits) of collections of such
peak projections  \cite{BNII}.      We will say more about peak projections around Lemma \ref{chux}.

In this paper we will sometimes use the word `cigar' for 
the wedge-shaped region consisting of numbers $re^{i \theta}$ with
argument $\theta$ such that $|\theta| < \rho$ (for some fixed small $\rho > 0$), 
which are also inside
the circle $|z - \frac{1}{2}| \leq \frac{1}{2}$.
If $\rho$ is small enough so that $|{\rm Im}(z)| < \epsilon/2$ for all $z$ in this
region,  then we will call this a `horizontal cigar of height $< \epsilon$ centered on the 
line segment $[0,1]$ in the $x$-axis'.

 By {\em numerical range}, we will mean the one defined by states, while the literature we quote
usually uses the one defined by vector states on $B(H)$.  However since the former range is the 
closure of the latter, as is well known, this will cause no difficulties.  
For any operator $T \in B(H)$ whose numerical range does not include strictly negative
numbers, and for any $\alpha \in [0,1]$, there is a well-defined `principal' root $T^\alpha$,
which obeys the usual law $T^\alpha T^\beta = T^{\alpha + \beta}$ if $\alpha + \beta \leq 1$
(see e.g.\ \cite{MP,LRS}).  If the numerical range   
is contained in a sector $S_\psi = \{ r e^{i \theta} : 0 \leq r , \, \text{and} \,  
-\psi \leq \theta \leq \psi \}$ where $0 \leq \psi < \pi$, then things are better still.  For fixed $\alpha \in (0,1]$
there is a constant $K > 0$ with $\Vert T^\alpha - S^\alpha \Vert \leq K \Vert  T - S \Vert^\alpha$ for operators $S, T$ with numerical range
in $S_\psi$ (see \cite{MP,LRS}).   Our operators $T$ will in fact be accretive (that is,
 $\psi \leq \frac{\pi}{2}$), 
and then these                 
powers obey the usual laws such as  $T^\alpha T^\beta = T^{\alpha + \beta}$
for all $\alpha, \beta > 0$,
$(T^\alpha)^\beta
= T^{\alpha \beta}$ for $\alpha \in (0,1]$ and any $\beta > 0$, 
and $(T^*)^\alpha  = (T^\alpha)^*$.     We shall see in Lemma \ref{rootf} that if  $\psi < \frac{\pi}{2}$ then $T
\in {\mathfrak c}_{B(H)}$.  The numerical range 
of $T^\alpha$ lies in $S_{\alpha \frac{\pi}{2}}$  for any $\alpha \in (0,1)$.  Indeed if  $n \in \Ndb$ then $T^{\frac{1}{n}}$ is the unique $n$th root of $T$ with numerical 
range in $S_{\frac{\pi}{2 n}}$.   See e.g.\ \cite[Chapter IV, Section 5]{NF}, \cite{Haase}, and \cite{LRS} 
for all of these facts.   Some of the following facts are no doubt also in the literature, since we do not know of a reference
we sketch short proofs.

\begin{lemma} \label{roots}   For an accretive operator $T \in B(H)$ we have:
\begin{itemize} \item [(1)]   $(cT)^\alpha  = c^\alpha T^\alpha$ for positive scalars $c$, and $\alpha \geq 0$.
\item [(2)]  $\alpha \mapsto T^{\alpha}$ 
is continuous on $(0,\infty)$.
\item [(3)]    $T^\alpha \in {\rm oa}(T)$,  the operator algebra generated by $T$, if $\alpha > 0$.
\end{itemize} 
\end{lemma}  \begin{proof}   (1) \ This is obvious if $\alpha = \frac{1}{n}$ for $n \in \Ndb$ by  the uniqueness of $n$th roots discussed above. 
In general it can be proved e.g.\ by a change of variable in the Balakrishnan representation for powers
(see e.g.\ \cite{Haase}), or by the continuity in (2).

(2) \ By a triangle inequality argument, and the inequality for 
$\Vert T^\alpha - S^\alpha \Vert$ above, we may assume that $T \in {\mathfrak c}_{B(H)}$.    By (1) we may assume that
$T \in  \frac{1}{2} {\mathfrak F}_{B(H)}$.    Define $$f(z) = ((1-z)/2)^{\alpha} - ((1-z)/2)^{\beta} \; ,  \qquad z \in \Cdb, |z| \leq 1.$$
Via the relation $T^\alpha T^\beta = T^{\alpha + \beta}$ above, we may assume that $\beta \in (0,1]$.
Fix such $\beta$.  By complex numbers one can show that  $|f(z)| \leq g(|\alpha - \beta|)$  on the unit disk,
for a function $g$ with $\lim_{t \to 0^+} \, g(t) = 0$.   By von Neumann's inequality,
  used as in \cite[Proposition 2.3]{BRII}, we have $$\Vert T^\alpha - T^\beta \Vert= \Vert f(1 - 2T) \Vert  \leq g(|\alpha - \beta|).$$ 
Now let $\alpha \to \beta$.  

(3) \ We proved this  in the second paragraph 
of \cite[Section 3]{BRII}  if $\alpha = \frac{1}{n}$ for $n \in \Ndb$.   Hence for $m \in \Ndb$ we have by the paragraph above
the lemma that $T^{\frac{m}{n}} = (T^{\frac{1}{n}})^m \in {\rm oa}(T)$.  The general case for $\alpha > 0$ then follows by the continuity in (2).
\end{proof} 
  
As in  \cite[Theorem 1]{MP} and \cite[Lemma 3.8]{BOZ}, if $\alpha \in (0,1)$ then there
exists a constant $K$ such that if $a, b \in  {\mathfrak r}_{B(H)}$
for a Hilbert space $H$, and $ab = ba$,  then
 $\Vert (a^{\alpha}  - b^\alpha) \zeta \Vert
\leq K \Vert (a-b) \zeta \Vert^\alpha$, for $\zeta \in H$.

\section{Positivity in operator algebras}

Let $A$ be an operator algebra, not necessarily approximately  unital for the present.
Note that ${\mathfrak r}_A = \{ a \in A : a + a^* \geq 0 \}$
is a closed cone in $A$, hence is Archimedean, but it is not 
proper (hence is what is sometimes called a {\em wedge}).  On the other hand
${\mathfrak c}_A = \Rdb_+ {\mathfrak F}_A$ is not closed in general, but it is a a proper cone
(that is, ${\mathfrak c}_A \cap (-{\mathfrak c}_A) = (0)$).
Indeed suppose $a \in {\mathfrak c}_A \cap (-{\mathfrak c}_A)$.  Then  $\Vert 1 - t a \Vert
\leq 1$ and $\Vert 1 + s a \Vert
\leq 1$ for some $s, t > 0$.  By convexity we may assume $s = t$ (by replacing them by
$\min \{ s , t \}$).   It is well known that in any
Banach algebra with an identity of norm $1$, the identity  is an extreme point of the ball.
Applying this in $A^1$ we deduce that  $a = 0$ as desired.

As we said earlier without proof, for
any operator algebra $A$,  $x \in {\mathfrak r}_A$  iff ${\rm Re}(\varphi(x)) \geq 0$ for all
states $\varphi$ of $A^1$.   Indeed, such $\varphi$
extend to states on $C^*(A^1)$.  
 So we may assume that $A$ is a
unital $C^*$-algebra, in which case the result is well known ($x + x^* \geq 0$ iff
$2 {\rm Re}(\varphi(x)) = \varphi(x+x^*) \geq 0$ for all
states $\varphi$).   We remark though that for an operator algebra which is not approximately unital, it is not true that $x \in {\mathfrak r}_A$  iff ${\rm Re}(\varphi(x)) \geq 0$ for all
states $\varphi$ of $A$, with states defined as in the introduction.  An example would be $\Cdb \oplus \Cdb$, with the second summand given the 
zero multiplication. 

The ${\mathfrak r}$-{\em ordering} is simply the order $\preccurlyeq$ induced by the above closed cone; that is
$b \preccurlyeq a$ iff $a - b \in {\mathfrak r}_A$.  
If $A$ is a subalgebra of an operator algebra $B$, it is clear from
a fact mentioned in the introduction (or at the start of \cite[Section 3]{BRII}) that the positivity of $a + a^*$ may be computed
with reference to any containing $C^*$-algebra, that ${\mathfrak r}_A \subset {\mathfrak r}_B$.
If $A, B$ are
approximately unital subalgebras of $B(H)$ then it follows from \cite[Corollary 4.3 (2)]{BRII}
that $A \subset B$ iff ${\mathfrak r}_A \subset {\mathfrak r}_B$.   As in \cite[Section 8]{BRI},
${\mathfrak r}_A$ contains no idempotents which are not orthogonal projections,
and no nonunitary isometries $u$ (since by the analogue 
of \cite[Corollary 2.8]{BRI} we would have $u u^* = s(u u^*) = s(u^* u) = I$).  
In \cite{BRII} it is shown that $\overline{{\mathfrak c}_A} = {\mathfrak r}_A$.

The following  
begins to illustrate the interesting order theory
that exists in an operator algebra $A$ and its generated $C^*$-algebra $B$.   Note particularly how the 
order theoretic results (3)--(7) flow out of the new `cofinality of $A$ in $B$ result'
 (item (2) or (2')).  See \cite{BOZ} (particularly Section 6 there) for
 more interesting connections to, and  remarkable relationships with, the 
classical theory of ordered linear spaces.    In Section 4 we shall see the 
relationship between  (2') and `noncommutative peak interpolation'.

\begin{theorem} \label{havin}  Let  $A$ be an operator algebra which generates a $C^*$-algebra
$B$, and let ${\mathcal U}_A$ denote the open unit ball $\{ a \in A : \Vert a \Vert < 1 \}$.  The following are equivalent:
\begin{itemize} \item [(1)]   $A$ is approximately unital.
 \item [(2)]  For any positive $b \in {\mathcal U}_B$ there exists  $a \in {\mathfrak c}_A$
with $b \preccurlyeq a$.
 \item [(2')]  Same as {\rm (2)}, but also $a \in \frac{1}{2}  {\mathfrak F}_A$ and nearly positive.
\item [(3)]   For any pair 
$x, y \in {\mathcal U}_A$ there exist   nearly positive
$a \in \frac{1}{2}  {\mathfrak F}_A$
with $x \preccurlyeq a$ and $y \preccurlyeq a$.
\item [(4)]    For any $b \in {\mathcal U}_A$  there exist   nearly positive
$a \in \frac{1}{2}  {\mathfrak F}_A$
with $-a \preccurlyeq b \preccurlyeq a$. 
\item [(5)] For any $b \in {\mathcal U}_A$  there exist 
$x, y \in  \frac{1}{2}  {\mathfrak F}_A$ 
with $b = x-y$.   
\item [(6)]  ${\mathfrak r}_A$ is a generating cone (that is, $A = {\mathfrak r}_A - {\mathfrak r}_A$). 
\item [(7)]  $A = {\mathfrak c}_A - {\mathfrak c}_A$. 
\end{itemize}  \end{theorem} 

\begin{proof}   (1) $\Rightarrow$ (2') \  Let $(e_t)$ be a cai for $A$ in
$\frac{1}{2} {\mathfrak F}_A$ (by Read's theorem stated in the Introduction).     
By \cite[2.1.6]{BLM}, $(e_t)$ is a cai for  $B$,
and hence so is $(e_t^*)$,
and  $f_t = {\rm Re} \, (e_t)$.   By the proof of Cohen's factorization theorem, 
as adapted in e.g.\ \cite[Lemma 4.8]{BOZ}, 
we may write  $b^2 = z w z$, where $0 \leq w \leq 1$ and  
$$z = \sum_{k=1}^\infty \, 2^{-k} \, f_{t_k} = {\rm Re}(\sum_{k=1}^\infty \, 2^{-k} \,
e_{t_k}),$$
where $\{ f_{t_k} \}$ are some of the $f_t$.   If $a = \sum_{k=1}^\infty \, 2^{-k} \,
e_{t_k} \in \frac{1}{2} {\mathfrak F}_A$, then $z = {\rm Re}(a)$.  
Then $b^2 \leq z^2$, so that $b \leq z$ and $b \preccurlyeq a$.    
We also have $b \preccurlyeq a^{\frac{1}{n}}$ for each $n \in \Ndb$ by \cite[Proposition 
4.7]{BBS},
which gives the `nearly positive' assertion.   

(2')  $\Rightarrow$ (3) \ By $C^*$-algebra theory there exists  positive $b \in {\mathcal U}_B$ with 
$x$ and $y$ `dominated' by $b$.  Then apply (2').

 (3) $\Rightarrow$ (4) \ Apply (3) to $b$ and $-b$.

(4)  $\Rightarrow$ (6) \  $b = \frac{a+b}{2} - \frac{a-b}{2} \in {\mathfrak r}_A - {\mathfrak r}_A$.

(6) $\Rightarrow$ (1) \ This is in \cite[Section 4]{BRII}, but we give a variant of the  argument.
First suppose that $A$ is a weak* closed subalgebra
of $B(H)$.    Each $x \in {\mathfrak r}_A$ has a 
support projection $p_x \in B(H)$ by the discussion in \cite[Section 3]{BBS}, which is 
just the weak* limit of $(x^{\frac{1}{n}})$, and hence is in $A$.  Then $p = \vee_{x \in {\mathfrak r}_A} \, p_x$ is in 
$A$, and for any $x \in {\mathfrak r}_A$ we have
$$p \, x = p \, s(x) \, x = s(x) \, x = x.$$
Since ${\mathfrak r}_A$ is generating, we have $px = x$ for all $x \in A$.   
Similarly, $x p = x$.   So $A$ is unital.    In the general case, we can use the fact from  theory of
ordered spaces \cite{AE2} that if the order in $A$ is generating, then the order in $A^*$ is normal, 
and then the order in $A^{**}$ is generating.  The latter forces $A^{**}$ to be unital, and hence $A$ is
approximately unital by e.g.\ \cite[Proposition 2.5.8]{BLM}.  

(1) $\Rightarrow$ (5) \  Apply \cite[Theorem 6.1]{BOZ}.

It  is obvious that (2') implies (2), and that (5) implies (7), which implies (6).

(2)  $\Rightarrow$ (6) \   If $a \in A$ then by $C^*$-algebra theory and (2) 
there exists $b \in B_+$ and $x \in {\mathfrak r}_A$ with $-x \preccurlyeq -b  \preccurlyeq a \preccurlyeq b \preccurlyeq x$.
Thus $a = \frac{a+x}{2} - \frac{x-a}{2} \in {\mathfrak r}_A - {\mathfrak r}_A$.  
\end{proof}

{\bf Remarks.}   1) \ One cannot expect to be able to choose the $a$ in (2) with
$\Vert a \Vert = \Vert b \Vert$.  Indeed, suppose that  $A = \{ f \in A(\Ddb) : f(1) = 0 \}$ and $B = \{ f \in  C(\Tdb) : f(1) = 0 \}$, with $b = 1$ on 
a nontrivial arc.  If $b \leq {\rm Re} (a) \leq |a| \leq 1$ on that arc, then ${\rm Re} (a) = a = 1$
on that arc too.  But this implies that $a =1$ always, a contradiction.

Similarly, in (3) one cannot replace ${\mathcal U}_A$ by Ball$(A)$, even if $A$ is a $C^*$-algebra (consider 
for example the universal nonunital $C^*$-algebra generated by two projections \cite{SiRa}).
However perhaps this (and also (4)) is possible if $B$ is commutative.   Some remarks on (5) may be found in 
\cite{BOZ} after Theorem 6.1.   
 
\smallskip

2) \ Another proof that (1) implies (2):  if $b \in B_+$ with $\Vert b \Vert < 1$
then it  is immediate from  \cite[Lemma 2.1]{Bnpi} that
 there exists $x \in -{\mathfrak F}_A$ such that $b \leq - x^* x - 2 \, {\rm Re} \, x$.   Hence $b \preccurlyeq a$, where $a = -2 x \in 2 {\mathfrak F}_A$.

This leads to an  quick proof  that (1) implies (2') if $b$ commutes
with Re$(a)$.  Namely, first choose $\epsilon > 0$
such that $(1+\epsilon) \, \Vert b \Vert < 1$.
Let $c = (1+\epsilon) \, b$,
and suppose that $m \in \Ndb$, and choose by the last 
paragraph $a \in 2 {\mathfrak F}_A$ with  
$c^m \leq {\rm Re} \, a$.  
Hence if $n \in \Ndb$ we have
$1 \leq {\rm Re} \, z$ where $z = (c^m + \frac{1}{n})^{-1} (a + \frac{1}{n})$.   It follows
from a result on p.\ 181 of \cite{Haase} that $1  \leq {\rm Re} \, z^{\frac{1}{m}}$.  
Thus $(c^m + \frac{1}{n})^{\frac{1}{m}} \leq {\rm Re} \, ((a + \frac{1}{n})^{\frac{1}{m}})$.
Letting $n \to \infty$ we obtain $c \leq {\rm Re} \, (a^{\frac{1}{m}})$.
For $m \geq m_0$ say, we have $$a^{\frac{1}{m}} = 4^{\frac{1}{m}} (\frac{a}{4})^{\frac{1}{m}}
 \in \frac{1+\epsilon}{2} {\mathfrak F}_A.$$  Dividing by $1+\epsilon$, and
taking $m$ large enough we obtain (2').    

\smallskip

3) \ Of course all parts of the theorem are trivial if $A$ is unital.

\subsection{Non-approximately unital operator algebras}

Most of the results in this Section apply to approximately unital operator algebras.  We offer a couple of results 
that are useful in applying the approximately unital  case to algebras with no approximate identity.  We 
will use the space $A_H = \overline{{\mathfrak r}_{A} A {\mathfrak r}_{A}}$ 
studied in \cite[Section 4]{BRII};
it is actually a  HSA in $A$ (and will be an ideal if $A$ is commutative).

\begin{corollary} \label{Ahasc2}  For any operator algebra $A$,
the largest approximately unital subalgebra of $A$ is
$$A_H = {\mathfrak r}_A - {\mathfrak r}_A = {\mathfrak c}_A - {\mathfrak c}_A.$$
In particular these spaces are  closed, and form a HSA of $A$.

If $A$ is a  weak* closed operator
algebra then $A_H = q A q$ where $q$ is the largest projection in $A$.
This is weak* closed.
 \end{corollary}

\begin{proof} 
In the language of \cite[Section 4]{BRII}, and using
\cite[Corollary 4.3]{BRII}, ${\mathfrak r}_A = {\mathfrak r}_{A_H}$, 
and the largest approximately unital subalgebra of $A$ is the HSA 
$$A_H  = {\mathfrak r}_{A_H} - {\mathfrak r}_{A_H}
= {\mathfrak r}_A - {\mathfrak r}_A,$$ using Theorem \ref{havin} (6).  
A similar argument works in the ${\mathfrak c}_A$ case,
with ${\mathfrak r}_{A_H}$ replaced by
${\mathfrak c}_{A_H}$ using Theorem \ref{havin} (7) and facts from
\cite[Section 4]{BRII} about ${\mathfrak F}_{A_H}$.

To see the final assertion, note that if $p$ is
as in the proof of (6) $\Rightarrow$ (1) in Theorem \ref{havin}, then certainly $q \leq p$
since $q = s(q) \in {\mathfrak r}_A$.
However $p \leq q$ since $p$ is a projection in $A$.  So $p = q$,
and this acts as the identity on ${\mathfrak r}_A - {\mathfrak r}_A = A_H$.   So $A_H \subset qA q$, and conversely
$q A q \subset   A_H$ since $A_H$ is a HSA,
or because $A_H$ is the largest (approximately) unital subalgebra
of $A$.  \end{proof}

\begin{lemma} \label{mnah}  Let $A$ be any operator algebra.  Then   for every $n \in \Ndb$, 
$$M_n(A_H) = M_n(A)_H \; , \; \; \; \; \;  {\mathfrak r}_{M_n(A)} = {\mathfrak r}_{M_n(A_H)}
 \; , \; \; \; \; \;  {\mathfrak F}_{M_n(A)} = {\mathfrak F}_{M_n(A_H)}$$
(these are the matrix spaces).  
\end{lemma}  

\begin{proof}   Clearly $M_n(A_H)$ is an approximately unital subalgebra of $M_n(A)$.  So $M_n(A_H)$ is contained in $M_n(A)_H$,
 since the latter is the largest approximately unital subalgebra of $M_n(A)$.    
To show that $M_n(A)_H \subset M_n(A_H)$ it suffices by 
Corollary \ref{Ahasc2} to show that ${\mathfrak r}_{M_n(A)} \subset M_n(A_H)$.  So  suppose that $a = [a_{ij}] \in M_n(A)$
with $a + a^* \geq 0$.  Then $a_{ii} + a_{ii}^* \geq 0$ for each $i$.  We also have 
$\sum_{i,j} \, \bar{z_i} \, (a_{ij} + a_{ji}^*) \, z_j \geq 0$ for all scalars 
$z_1, \cdots , z_n$.   So $\sum_{i,j} \, \bar{z_i} \, a_{ij} z_j \in {\mathfrak r}_A$.
 Fix an $i, j$, which we will assume to be $1, 2$ for simplicity.
Set all $z_k = 0$ if $k \notin \{ i, j \} = \{ 1, 2 \}$, to deduce 
$$\bar{z_1} z_2 a_{12}  + \bar{z_2} z_1 a_{21} = \sum_{i,j = 1}^2 \, \bar{z_i} \, a_{ij} z_j \, - \, (|z_1|^2 a_{11} +
|z_2|^2 a_{22}) \in {\mathfrak r}_A - {\mathfrak r}_A = A_H.$$
Choose $z_1 = 1$; if  
$z_2 = 1$ then $a_{12}  + a_{21} \in A_H$, while if $z_2 = i$ then 
$i(a_{12}  -  a_{21}) \in A_H$.    So $a_{12}, a_{21} \in A_H$.  A similar argument 
shows  that $a_{ij}  \in A_H$ for all $i, j$.  Thus $M_n(A_H) = M_n(A)_H$, from which we deduce
by \cite[Corollary 4.3 (1)]{BRII} that
$${\mathfrak r}_{M_n(A)} = {\mathfrak r}_{M_n(A)_H} = {\mathfrak r}_{M_n(A_H)}.$$
Similarly ${\mathfrak F}_{M_n(A)} = {\mathfrak F}_{M_n(A)_H}
= {\mathfrak F}_{M_n(A_H)}$.
\end{proof}

The last result is used in \cite{BBS}.

 If $S \subset {\mathfrak r}_A$, for an operator algebra
$A$,  and if $xy = yx$
for all $x, y \in S$, write oa$(S)$ for the smallest closed subalgebra of $A$ containing $S$.

\begin{proposition} \label{commt}  If $S$ is any subset of ${\mathfrak r}_A$ for an operator algebra
$A$,  
then ${\rm oa}(S)$ has a cai.
\end{proposition}

\begin{proof}   Let $C = {\rm oa}(S)$.  Then  ${\mathfrak r}_C = C \cap {\mathfrak r}_A$, so that 
$$C \subset \overline{{\mathfrak r}_C \, C \, {\mathfrak r}_C} = C_H \subset C.$$   Hence $C = C_H$ which is
approximately unital.   
\end{proof}

\subsection{The ${\mathfrak F}$-transform and existence of an increasing approximate identity}

In \cite{BRII} the sets  $\frac{1}{2}{\mathfrak F}_{A}$ and ${\mathfrak r}_{A}$ were related by a certain 
transform.  We now establish a few more basic properties of this transform. 
The Cayley transform $\kappa(x) = (x-I)(x+I)^{-1}$ 
of an accretive $x \in A$ exists since $-1 \notin {\rm Sp}(x)$, and
is well known to be a contraction.  Indeed it is well known (see e.g.\ \cite{NF}) 
that if $A$ is unital then the Cayley transform
maps ${\mathfrak r}_A$ bijectively onto the set of contractions in $A$ whose 
spectrum does not contain $1$, and the inverse transform is
$T \mapsto (I+T) (I - T)^{-1}$.  The Cayley transform maps the accretive elements $x$
with ${\rm Re}(x) \geq \epsilon 1$ for some $\epsilon > 0$, onto the set 
of elements $T \in A$ with $\Vert T \Vert < 1$ (see e.g.\ 2.1.14 in \cite{BLM}). 
 The ${\mathfrak F}$-transform ${\mathfrak F}(x) = 1 - (x+1)^{-1} = x (x+1)^{-1}$ may be written
as ${\mathfrak F}(x) = \frac{1}{2} (1 + \kappa(x))$.  Equivalently,
 $\kappa(x) = -(1 - 2 {\mathfrak F}(x))$.   

\begin{lemma} \label{Font}   For any operator algebra $A$, the  ${\mathfrak F}$-transform 
 maps ${\mathfrak r}_{A}$ bijectively onto the set of elements of $\frac{1}{2}{\mathfrak F}_{A}$
of norm $< 1$.   Thus ${\mathfrak F}({\mathfrak r}_A) = {\mathcal U}_A \cap \frac{1}{2}{\mathfrak F}_{A}$.  \end{lemma}  \begin{proof}  
First assume that  $A$ is  unital. By the last equations  ${\mathfrak F}({\mathfrak r}_A)$ is 
contained in  the set of elements of $\frac{1}{2}{\mathfrak F}_{A}$
whose spectrum does not contain $1$.  
The inverse of the ${\mathfrak F}$-transform
on this domain is 
$T (I-T)^{-1}$.   To see for example that $T (I-T)^{-1} \in {\mathfrak r}_A$ if $T \in \frac{1}{2}{\mathfrak F}_{A}$
note that 2Re$(T (I-T)^{-1})$ equals
$$(I-T^*)^{-1}(T^*(I-T) + (I-T^*) T) (I-T)^{-1} = (I-T^*)^{-1}(T + T^* - 2T^* T) (I-T)^{-1}$$
which is positive since $T^* T$  is dominated by Re$(T)$ if $T \in \frac{1}{2}{\mathfrak F}_{A}$.  
 Hence for any (possibly nonunital) operator algebra $A$ the  ${\mathfrak F}$-transform 
 maps ${\mathfrak r}_{A^1}$ bijectively onto the set of elements of $\frac{1}{2}{\mathfrak F}_{A^1}$
whose spectrum does not contain $1$.   However this equals the set of elements of $\frac{1}{2}{\mathfrak F}_{A^1}$
of norm $< 1$.     Indeed if $\Vert  {\mathfrak F}(x) \Vert = 1$ then $\Vert  \frac{1}{2} (1 + \kappa(x)) \Vert = 1$, and so $1 - \kappa(x)$ 
is not invertible by \cite[Proposition 3.7]{ABS}.   Hence $1 \in  {\rm Sp}_{A^1}(\kappa(x))$ and $1 \in  {\rm Sp}_A({\mathfrak F}(x))$.
Since ${\mathfrak F}(x) \in A$ iff $x \in A$, we are done.   \end{proof}

Thus in some sense we can identify ${\mathfrak r}_{A}$  with the strict contractions in $\frac{1}{2}{\mathfrak F}_{A}$.  This for example induces an order on this set of strict contractions.   

We recall that the positive part of the  open unit ball of a $C^*$-algebra
is a directed set, and indeed is a  net which is a positive cai for $B$ (see e.g.\ \cite{Ped}).  The following generalizes this to operator algebras:

\begin{proposition} \label{isnet}  If $A$ is an approximately unital  operator algebra, then
${\mathcal U}_A \cap \frac{1}{2} {\mathfrak F}_A$ is a directed set in the $\preccurlyeq$ ordering,
and with this ordering ${\mathcal U}_A \cap \frac{1}{2} {\mathfrak F}_A$  is an increasing 
cai for $A$.  \end{proposition} \begin{proof}  We know ${\mathfrak F}({\mathfrak r}_A) = 
{\mathcal U}_A \cap \frac{1}{2} {\mathfrak F}_A$ by Lemma \ref{Font}.  
By Theorem \ref{havin} (3), ${\mathcal U}_A \cap \frac{1}{2} {\mathfrak F}_A$ is directed by
$\preccurlyeq$.   So we may view ${\mathcal U}_A \cap \frac{1}{2} {\mathfrak F}_A$
as a net $(e_t)$.   Given $x \in \frac{1}{2} {\mathfrak F}_A$, choose
$n$ such that $\Vert {\rm Re}(x^{\frac{1}{n}}) \, x - x \Vert < \epsilon$ (note that as in
the first few lines of the proof of Theorem \ref{havin}, $({\rm Re}(x^{\frac{1}{n}}))$ is a cai for $C^*({\rm oa}(x))$. 
If $z \in {\mathcal U}_A \cap \frac{1}{2} {\mathfrak F}_A$ with 
$x^{\frac{1}{n}} \preccurlyeq z$  then   
$$x^* |1-z|^2 x \leq x^* (1 - {\rm Re} \, (z)) x \leq 
x^* (1 - {\rm Re}(x^{\frac{1}{n}})) x \leq \epsilon.  $$ 
Thus $e_t \, x \to x$ for all $x \in \frac{1}{2} {\mathfrak F}_A$.   \end{proof}  

Note that ${\mathcal U}_A \cap {\mathfrak r}_A$ is 
directed, by Theorem \ref{havin} (3),
but we do not know if it is a cai in this ordering.

The following is a variant of \cite[Corollary 2.10]{BOZ}:

\begin{corollary}  \label{proz} Let $A$ be an
approximately unital  operator algebra, and $B$ a $C^*$-algebra generated by $A$.
If $b \in B_+$ with $\Vert b \Vert < 1$ then 
there is an increasing cai for $A$ in $\frac{1}{2} {\mathfrak F}_A$, every term of which dominates $b$
(where `increasing' and `dominates'
are in the $\preccurlyeq$ ordering).
\end{corollary} \begin{proof}  
Since ${\mathcal U}_A \cap \frac{1}{2} {\mathfrak F}_A$ is a directed 
set,  $\{ a \in {\mathcal U}_A \cap \frac{1}{2} {\mathfrak F}_A
: b \preccurlyeq a \}$ is a subnet of the 
increasing cai in the last result.    \end{proof}

We remark that any  operator algebra $A$ with a countable cai, and in particular
any separable approximately unital $A$, has a commuting  cai which is increasing
(for the $\preccurlyeq$ ordering), and 
also in $\frac{1}{2} {\mathfrak F}_A$
and nearly positive.   Namely, by \cite[Corollary 2.18]{BRI} we have $A = \overline{xAx}$ for some $x \in  \frac{1}{2} {\mathfrak F}_A$,
so that $(x^{\frac{1}{n}})$ is a commuting  cai which is increasing by \cite[Proposition
4.7]{BBS}.    For a related fact see Lemma \ref{cssigc} below.

\subsection{Real positive maps and real states}  An $\Rdb$-linear $\varphi : A \to \Rdb$ (resp.\ $\Cdb$-linear $T : A \to B$) will be said to be 
{\em real positive} if $\varphi({\mathfrak r}_A) \subset [0,\infty)$ (resp.\ $T({\mathfrak r}_A)
\subset {\mathfrak r}_B$).   By the usual trick, for any $\Rdb$-linear $\varphi : A \to \Rdb$, there
is a unique $\Cdb$-linear $\tilde{\varphi} : A \to \Cdb$  with Re $\, \tilde{\varphi} = \varphi$,
and clearly $\varphi$ is real positive (resp.\ bounded) iff $\tilde{\varphi}$ is real positive (resp.\ bounded). 

\begin{corollary}  \label{pr} Let $A$ be an 
approximately unital  operator algebra, and $B$ a $C^*$-algebra generated by $A$.
Then every 
real positive $\varphi : A \to \Rdb$ extends to a real positive real functional on $B$.  Also, $\varphi$ is bounded.  
\end{corollary} \begin{proof}    Theorem \ref{havin} (2) says that the ordering in $A$ is  dominating or `cofinal' in $B$ in the language 
of ordered spaces (see  e.g.\ \cite{Jameson}).  The first assertion is a well known consequence in the theory of ordered spaces
of this cofinal property 
(see e.g.\ \cite{Day}  or \cite[Theorem 1.6.1]{Jameson}).   Similarly the final assertion follows from a general principle for 
an ordered Banach space $(X, \leq)$ whose order is generating: if $f : X \to \Rdb$ is positive but 
(by way of contradiction) unbounded then by a theorem of Ando $f$ is unbounded on 
${\rm Ball}(X_+)$.  So there exist $x_k \in X_+$ of norm $\leq$ but with $f(x_k)> 2^k$.
So $n < \sum_{k=1}^n \, 2^{-k} \, f(x_k) \leq f(\sum_{k=1}^\infty \, 2^{-k} \, x_k)$ for all $n$, a contradiction.   
\end{proof}

\begin{corollary}   Let $T : A \to B$ be a $\Cdb$-linear  map between 
approximately unital  operator algebras,  and suppose that  $T$ is 
real positive (resp.\ suppose that the $n$th matrix amplifications $T_n$ are each  real positive (cf.\ 
{\rm \cite[Definition 2.1]{BBS}})).   Then $T$ is bounded (resp.\ completely bounded).  
\end{corollary}  \begin{proof}  First suppose that $B = \Cdb$.  Then Re$ \, T$ is real positive,
hence bounded by Corollary \ref{pr}.  It is then obvious that $T$ is bounded.  

In the general case, we can assume $B$ is a unital $C^*$-algebra.
Let  $\psi \in S(B)$, and $\varphi  = \psi \circ T$.  Then $\varphi$ is real positive, hence
bounded.  Thus there exists a constant $K$ such that for all $x \in {\rm Ball}(A)$ we have  
$|\psi(T(x))| = |\varphi(x) | \leq K$.     By the `Jordan decomposition' in $B^*$, 
it follows that $|\psi(T(x))| | \leq 4K$ for all $\psi \in {\rm Ball}(B^*)$. 
Thus $T$ is bounded.   In the `respectively' case,  applying the above at each
matrix level shows that the $n$th amplifications $T_n$ are each bounded.  The proof in
\cite[Section 2]{BBS} shows that $T$ extends to a completely positive map 
on an operator system, and it is known that completely positive maps are completely bounded. 
\end{proof}

{\bf Remark.}  It follows from this that in the `Extension and Stinespring dilation theorem for real completely positive maps'  from \cite{BBS}, 
it is unnecessary to assume that the RCP maps defined in \cite[Definition 2.1]{BBS}
are (completely) bounded.   One only needs $T$ to be linear and 
real positive, and similarly at each matrix level.  

\bigskip

 We will write ${\mathfrak c}^{\Rdb}_{A^*}$ for the real dual cone of ${\mathfrak r}_A$, the 
set of continuous $\Rdb$-linear $\varphi : A \to \Rdb$ such that $\varphi({\mathfrak r}_A) \subset [0,\infty)$.
Since $\overline{{\mathfrak c}_A} = {\mathfrak r}_A$ this is also the real dual cone of ${\mathfrak c}_A$.    It is a   proper cone for if  $\rho, -\rho \in {\mathfrak c}^{\Rdb}_{A^*}$ then
$\rho(a) = 0$ for all $a \in {\mathfrak r}_A$.  Hence $\rho = 0$
by the fact above that the norm closure of ${\mathfrak r}_A - {\mathfrak r}_A$
is $A$.

\begin{lemma}  \label{dualc}    Suppose that $A$ is an approximately unital operator algebra.
 The real dual cone ${\mathfrak c}^{\Rdb}_{A^*}$  equals 
$\{ t \, {\rm Re}(\psi) : \psi \in S(A) , \, t \in [0,\infty) \}$. 
It also equals the set of restrictions to $A$ of the real parts of
positive functionals on any $C^*$-algebra containing 
(a copy of) $A$ as a closed subalgebra.   The prepolar of ${\mathfrak c}^{\Rdb}_{A^*}$, which equals its
real predual cone,  is  ${\mathfrak r}_{A}$;  
and  the polar of ${\mathfrak c}^{\Rdb}_{A^*}$,  which equals its
real dual cone,  is  ${\mathfrak r}_{A^{**}}$.   Thus the second dual cone of 
${\mathfrak r}_A$ is ${\mathfrak r}_{A^{**}}$, and hence 
 ${\mathfrak r}_A$ is weak* dense in ${\mathfrak r}_{A^{**}}$. \end{lemma}  

\begin{proof}   This is proved in \cite{BOZ} in a more general setting,
but there is a simpler proof in our case.
By Corollary \ref{pr}, every real positive $\varphi : A \to \Rdb$ extends to a real positive real functional 
on $B$, and the latter is the real part of a $\Cdb$-linear  real positive functional $\psi$ 
on $B$.   Clearly $\psi$ is positive in the usual sense, 
and hence $\psi$ is a positive multiple of a state on $B$.
Restricting to $A$, we see that  $\varphi$  is the real part of a positive multiple of a state on $A$.  
Thus  $${\mathfrak c}^{\Rdb}_{A^*} = \{ t \, {\rm Re}(\psi) : \psi \in S(A) , \, t \in [0,\infty) \}.$$
In any $C^*$-algebra $B$ it is well known that $b \geq 0$ iff $\varphi(b) \geq 0$ for all states $\varphi$ of $B$.
Hence $a \in  {\mathfrak r}_{A} =  A \cap  {\mathfrak r}_{B}$ iff  $2 \, {\rm Re} \, \varphi(a) = \varphi(a + a^*)  \geq 0$ for all states $\varphi$, and so iff $a \in ({\mathfrak c}^{\Rdb}_{A^*})_\circ$.   
The polar of ${\mathfrak c}^{\Rdb}_{A^*}$ is  
$$\{ \eta \in A^{**} : {\rm Re} \,  \eta(\psi) \geq 0 \; \text{for all} \; \psi \in S(A) \}
= {\mathfrak r}_{B^{**}} \cap A^{**} = {\mathfrak r}_{A^{**}},$$
since $${\mathfrak r}_{B^{**}} = \{ \eta \in B^{**} : {\rm Re} \, \eta(\psi) \geq 0 \; \text{for all} \; \psi \in S(B) \}.$$  
So the real bipolar $({\mathfrak r}_A)^{\circ \circ} = {\mathfrak r}_{A^{**}}$. 
 By the bipolar theorem, ${\mathfrak r}_A$ is weak* dense in ${\mathfrak r}_{A^{**}}$.   \end{proof}

We remark that the last several results have some depth; indeed one can show that they
are each essentially equivalent 
to Read's theorem on approximate identities (and can be used to give a more order theoretic
proof of that result).

We give some consequences to the theory
of  real states.   
A {\em real state} on approximately unital operator algebra $A$ will be a contractive  $\Rdb$-linear $\Rdb$-valued functional on $A$
such that $\varphi(e_t) \to 1$ for some cai $(e_t)$ of $A$.  This is equivalent to
$\varphi^{**}(1) = 1$, where $\varphi^{**}$ is the canonical 
$\Rdb$-linear extension to $A^{**}$, and $1$ is the identity of 
$A^{**}$ (here we are using the canonical identification between real second duals
and complex second duals of a complex Banach space \cite{Li}).  Hence $\varphi(e_t) \to 1$ for every
cai $(e_t)$ of $A$.  

Since we can identify $A^1$ with $A + \Cdb 1_{A^{**}}$ if we like, by the 
last paragraph it follows that  real states of $A$ extend to real states of $A^1$, hence by the 
Hahn-Banach theorem they extend to real states of $C^*(A^1)$.
We claim that a  real state $\psi$ on a $C^*$-algebra $B$ is positive on $B_+$, and
is zero  on $i B_+$.
To see this, we may assume that $B$ is a von Neumann algebra
(by extending the state to its second dual similarly to as in the last paragraph).  For any projection $p \in B$,
$C^*(1,p) \cong \ell^\infty_2$, and it is an easy exercise to see that 
real states on $\ell^\infty_2$ are positive on $(\ell^\infty_2)_+$
and are zero on $i (\ell^\infty_2)_+$.
Thus $\psi(p) \geq 0$ and $\psi(ip) = 0$ for any projection $p$, hence 
$\psi$ is positive on $B_+$  and zero  on $i B_+$ by the Krein-Milman theorem.   

We deduce:

\begin{corollary}  \label{duals}  Real states on an approximately unital operator algebra
$A$ are in ${\mathfrak c}^{\Rdb}_{A^*}$.   Indeed real states  are just the real parts of ordinary states on $A$.
\end{corollary}

\begin{proof}    Certainly 
the real part of an ordinary state is a real state.  If 
$\varphi$ is a real state on $A$, 
if $a + a^* \geq 0$, and if $\tilde{\varphi}$ is the 
real state extension above to $B = C^*(A^1)$, 
 then $$\varphi(a) = \frac{1}{2} \tilde{\varphi}(a + a^*) + 
\frac{1}{2} \tilde{\varphi}(-i \cdot i (a - a^*)) = 
\frac{1}{2} \tilde{\varphi}(a + a^*) \geq 0 ,$$ since 
$i(a-a^*) \in B_{\rm sa} = B_+ - B_+$, and $\tilde{\varphi}(i(B_+ - B_+)) = 0$, as we said 
above.   So $\varphi \in {\mathfrak c}^{\Rdb}_{A^*}$.  By  \cite[Corollary 6.3]{BOZ} we have 
$\varphi$ is the real part of a quasistate of $A$, and it is easy to see that the latter must be a state. 
 \end{proof}

\begin{corollary}  \label{uext}  Any real state on an approximately unital closed 
subalgebra $A$ of an approximately unital  operator algebra $B$ extends 
to a real state on $B$.  If $A$ is a HSA in $B$ then this extension is unique.
\end{corollary}

\begin{proof}
The first part  
is as in \cite[Proposition 3.1.6]{Ped}.
Suppose that  $A$ is a HSA in $B$ and that $\varphi_1, \varphi_2$ 
are real states  on $B$ extending a real state on $A$.  By the above we may
write $\varphi_i = {\rm Re} \, \psi_i$ for ordinary states  on $B$.
Since $\varphi_1 = \varphi_2$ on $A$ we have 
$\psi_1 = \psi_2$ on $A$.  Hence $\psi_1 = \psi_2$ on $B$ by 
\cite[Theorem 2.10]{BHN}.  So $\varphi_1 = \varphi_2$ on $B$.
  \end{proof}

\subsection{Principal $r$-ideals} \label{fgsect}

In 
the predecessor to this paper (\cite{BRIII}),  we proved several facts about principal and algebraically 
finitely generated $r$-ideals, and these were generalized to Banach algebras in \cite{BOZ}
with essentially the same proofs.   The main difference  is that in \cite{BOZ} one always had the 
condition that $A$ be approximately unital, whose purpose was simply so that 
${\mathfrak r}_A$ makes sense.   For operator algebras ${\mathfrak r}_A$ always makes sense, so that
one can delete  `approximately unital' in the statements of 3.21--3.25 in \cite{BOZ}.  One may also 
replace `idempotent' by `projection' in those results, since for operator algebras the support 
$s(x)$ is a projection  for $x \in {\mathfrak r}_A$.     One may also delete the word 
`left' in  \cite[Corollary 3.25]{BOZ} since a left identity is a two-sided identity if $A$ is approximately unital
(since $e e_t = e_t \to e$ for the cai $(e_t)$).
Moreover, the proofs show that
all of Theorem {\rm 3.2} of \cite{BRI} is valid for 
$x \in {\mathfrak r}_A$.   Similarly, the proof of \cite[Corollary 4.7]{BOZ} gives 

\begin{corollary} \label{Aha4}   Let $A$ be an operator algebra.   A closed $r$-ideal $J$ in $A$ 
is algebraically finitely generated as a right module over $A$   iff   $J  = eA$ for a projection $e \in A$.  This
is also equivalent to $J$ being algebraically finitely generated as a right module over $A^1$.
\end{corollary}

\subsection{Roots of accretive elements}

\begin{lemma} \label{rootf0}   Suppose that  $B$ is a $C^*$-algebra in its universal representation, so that $B^{**} \subset B(H)$ as
a von Neumann algebra containing $I_H$.  Let $x \in \frac{1}{2} \, {\mathfrak F}_B$ and let $s(x)$ be its support 
projection, viewed in $B(H)$.
Then $x^{\frac{1}{n}} \to s(x)$ in the strong operator topology.  \end{lemma} 

  \begin{proof}    If $\zeta  \in H$, and $a_n =  x^{\frac{1}{n}}$ then $a_n \in \frac{1}{2} \, {\mathfrak F}_B$  
by \cite[Proposition 2.3]{BRI}.   Hence $a_n^* a_n \leq {\rm Re} \, (a_n)$, and 
$$\Vert (a_n - s(x)) \zeta \Vert^2 =  \langle   (a_n^* a_n  - 2 {\rm Re} \, (a_n) +  s(x)) \zeta ,  \zeta  \rangle
\leq   \langle (s(x)   - {\rm Re} \, (a_n) )  \zeta ,  \zeta  \rangle \to 0 ,$$
since $a_n$, and hence $a_n^*$ and ${\rm Re} \, (a_n)$, converges weak* to $s(x)$.
\end{proof}  
 
\begin{lemma} \label{rootf}   Let $A$ be an operator algebra, and $x \in A$. 
\begin{itemize} \item [(1)]   If the numerical range of $x$ is contained in a sector
$S_{\rho}$ for $\rho < \frac{\pi}{2}$ 
 (see notation above
Lemma {\rm \ref{roots}}), then  $x/\Vert {\rm Re}(x) \Vert
\in \frac{\sec^2 \rho}{2} \, {\mathfrak F}_A$. So $x  \in {\mathfrak c}_A$.
\item [(2)]   
If $x \in {\mathfrak r}_A$ then 
$x^\alpha \in {\mathfrak c}_A$ for any $\alpha \in (0,1)$.  \end{itemize}
In particular, the elements of $A$ which are sectorial of angle $< \frac{\pi}{2}$ are
a dense subcone of ${\mathfrak c}_A$.
\end{lemma}  \begin{proof}     (1) \ Write $x = a + ib$,
for positive $a$ and selfadjoint $b$ in a containing $B(H)$.  By the argument in the proof
of \cite[Lemma 8.1]{BRI},  there exists a selfadjoint $c \in B(H)$
with $b = a^{\frac{1}{2}} c a^{\frac{1}{2}}$ and $\Vert c \Vert \leq \tan \rho$.
Then $x = a^{\frac{1}{2}} (1 + ic) a^{\frac{1}{2}}$, and 
$$x^* x = a^{\frac{1}{2}} (1 + ic)^* a (1 + ic) a^{\frac{1}{2}} \leq
C a.$$
By the $C^*$-identity $\Vert (1 + ic)^* a (1 + ic) \Vert$ equals
$$\Vert a^{\frac{1}{2}} (1 + ic) (1 + ic)^* a^{\frac{1}{2}}  \Vert
\leq \Vert a \Vert (1 + \Vert c \Vert^2) \leq
 \Vert a \Vert (1 + \tan^2 \rho) = \Vert a \Vert  \sec^2 \rho.$$ 
So we can take $C = \Vert a \Vert \, \sec^2 \rho$.  Saying that 
$x^* x \leq
C {\rm Re}(x)$ is the same as saying that $x \in \frac{C}{2} {\mathfrak F}_A$.

(2) \ This follows from (1) since in this case the numerical range of $x^\alpha$ is contained in 
a sector $S_{\rho}$ with
$\rho  < \frac{\pi}{2}$.  

The final assertion  follows from (1), and from the facts from the Introduction
 that $x = \lim_{t \to 1^-} \, x^t$
and that $x^t$ is sectorial of angle $< \frac{\pi}{2}$ if $0 < t < 1$.   \end{proof}  

{\bf Remark.}  The last result is related to the remark before \cite[Lemma 8.1]{BRI}.

\medskip

Of course $\Vert {\rm Im}(x^{\frac{1}{n}}) \Vert \to 0$ as $n \to \infty$, for $x \in {\mathfrak r}_A$ (as is clear e.g.\ from the above, or from the computation in the centered
line on the second page of our paper).

\begin{lemma} \label{vpow}  If $a \in {\mathfrak r}_A$ for an operator algebra $A$,
and $v$ is a partial isometry in any
 containing $C^*$-algebra $B$ with $v^* v = s(a)$, then
$v a v^* \in {\mathfrak r}_B$ and
$(v a v^*)^r = v a^r v^*$ if $r \in (0,1) \cup \Ndb$.
\end{lemma}

\begin{proof}   This is clear if $r  = k \in \Ndb$.   It is also clear that $v a v^* \in {\mathfrak r}_B$.
We will use the Balakrishnan representation above to check that $(v a v^*)^r = v a^r v^*$ if $r \in (0,1)$ (it can also 
be deduced from the ${\mathfrak F}_A$ case in \cite{BNII}). 
Claim: $(t + v a v^*)^{-1} v a v^* = v (t + a)^{-1} a v^*$.    Indeed since $v^* v a = a$ we have 
$$(t + v a v^*) v (t + a)^{-1} a v^* 
= v (t+a) (t + a)^{-1} a v^* = v a v^*,$$
proving the Claim.  Hence for any $\zeta, \eta \in H$
we have
$$\langle (t + v a v^*)^{-1} v a v^* \zeta, \eta \rangle
= \langle  v (t + a)^{-1} a v^*  \zeta, \eta \rangle 
= \langle  (t + a)^{-1} a v^*  \zeta, v^* \eta \rangle.$$   
Hence by the Balakrishnan representation $\langle (v a v^*)^r \zeta, \eta \rangle$ equals
$$\frac{\sin(r \pi)}{\pi} \int_0^\infty \, t^{r - 1} \, \langle
 (t + vav^*)^{-1} v a v^* \zeta , \eta \rangle \, dt = 
\frac{\sin(r \pi)}{\pi} \int_0^\infty \, t^{r - 1} \, \langle  (t + a)^{-1} a v^*  \zeta, v^* \eta \rangle \, dt,$$
which equals $\langle v a^r v^* \zeta, \eta \rangle$, as desired.
\end{proof}

The last result generalizes \cite[Lemma 1.4]{BNI}.   With the last few results  in hand, and \cite[Lemma 3.6]{BOZ},   it appears that 
all of the results in \cite{BNI} stated in terms of  ${\mathfrak F}_A$ (or $\frac{1}{2}  {\mathfrak F}_A$ or ${\mathfrak c}_A$),
should generalize without problem to the ${\mathfrak r}_A$ case.  We admit that we have not yet carefully checked every part of every result in 
\cite{BNI}  
  for this though, but hope to
in forthcoming work.     

\subsection{Concavity, monotonicity, and operator inequalities}  The usual operator
concavity/convexity results for $C^*$-algebras seem to fail for the
${\mathfrak r}$-ordering.  That is, results of the type in
\cite[Proposition 1.3.11]{Ped} and its proof fail.
Indeed, functions like  Re$(z^{\frac{1}{2}}),
{\rm Re}(z (1+z)^{-1}), {\rm Re}(z^{-1})$
are not operator concave or convex,  even for
operators  $x, y \in \frac{1}{2} {\mathfrak F}_A$.  In
fact this fails even in the simplest case $A = \Cdb$,
taking $x = \frac{1}{2}, y = \frac{1+i}{2}$.  Similar remarks hold for `operator monotonicity'
with respect to the ${\mathfrak r}_A$-ordering for these functions.

For   the 
${\mathfrak r}$-ordering, one way one can often prove operator inequalities, or that something is increasing, is via the 
functional calculus, as follows.  

\begin{lemma} \label{sfex}  Suppose that $A$ is a unital operator algebra and $f, g$ are functions in the disk algebra,
 with ${\rm Re} \, (g-f) \, \geq \, 0$ on the closed unit disk.  Then $f(1-x) \preccurlyeq g(1-x)$ for $x \in {\mathfrak F}_A$.  
\end{lemma}

\begin{proof}    Here e.g.\ $f(1-x)$  is the `disk algebra functional calculus', arising from von Neumann's inequality for the contraction $1-x$.   
The result follows by  \cite[Proposition 3.1, Chapter IV]{NF}  applied to $g-f$.  
\end{proof}

A good illustration of this principle is the proof at the end of
 \cite{BBS} that for any  $x \in \frac{1}{2} {\mathfrak F}_A$, the sequence
$({\rm  Re}(x^{\frac{1}{n}}))$ is increasing.   The last fact is another example of  $\frac{1}{2}   {\mathfrak F}_A$ behaving better than
${\mathfrak r}_A$:  for contractions $x \in {\mathfrak r}_A$, we do not in
general have $({\rm Re} (x^{1/m}))$ increasing  with $m$.
The matrix example 
$$\left[ \begin{array}{ccl} 1 & i \\ i & 0 \end{array} \right]$$
(communicated to us by Christian Le Merdy) will demonstrate this.    This example also shows
that one need not have $||x^{1/m}|| \leq ||x||^{1/m}$ for $x \in {\mathfrak r}_A$, so that one can 
have $x \in {\mathfrak r}_A \cap {\rm Ball}(A)$ but $||x^{1/m}|| > 1$.
  However one can show that
for any $x \in {\mathfrak r}_A$ there exists a constant $c > 0$ such that 
$({\rm Re} ((x/c)^{1/m}))_{m \geq 2}$ is  increasing  with $m$.   Indeed if 
$c = (2 \Vert {\rm Re}(x^{\frac{1}{2}}) \Vert)^2$, then by Lemma \ref{rootf} (2) we have 
$(x/c)^{\frac{1}{2}} \in \frac{1}{2} {\mathfrak F}_A$.
Thus ${\rm Re} ((x/c)^{t})$ increases as $t \searrow 0$ 
(see the proof of the \cite[Proposition 3.4]{BBS}),
from which the desired assertion  follows.

Finally, we clarify a few imprecisions in a couple 
 of the positivity results in \cite{BRI, BRII}.   At the end of Section 4
of \cite{BRII}, states on a nonunital algebra should probably also be assumed to have norm 1  (although the arguments there do not need this). 
 In \cite[Proposition 4.3]{BRI} we should have explicitly stated the hypothesis that $A$ is
approximately unital.  There are some small typo's in the proof of \cite[Theorem 2.12]{BRI}
but the reader should have no problem correcting these.

\section{Strictly real positive elements} \label{morp2} 

An element $x$ in $A$ with
${\rm Re}(\varphi(x)) > 0$ for all states on $A^1$
whose restriction 
to $A$ is nonzero, will be called 
{\em strictly real positive}.   Such $x$ are in $ {\mathfrak r}_A$.
This includes the $x \in A$ with Re$(x)$ strictly positive in some 
$C^*$-algebra generated by $A$.  
  If $A$ is approximately unital, then these conditions are in fact   equivalent, 
as the next result shows.
Thus the definition of  strictly real positive here generalizes the  definition 
given in \cite{BRI} for approximately unital operator algebras.

\begin{lemma} \label{secto}   Let $A$ be an approximately unital operator algebra, which 
generates a  $C^*$-algebra $C^*(A)$. 
An element $x \in A$ is strictly real positive in the sense above 
iff ${\rm Re}(x)$ is strictly positive in  
$C^*(A)$.   \end{lemma}  \begin{proof}  The one direction 
 follows because any state 
on $A^1$ whose restriction 
to $A$ is nonzero, extends to a state on $C^*(A)^1$ which is nonzero on 
$C^*(A)$.  The restriction to $C^*(A)$ of the latter state is a positive multiple of a state.

For the other direction 
recall that we showed in the introduction that  
any state on  $C^*(A)$ gives rise to a state on $A^1$.
Since any cai of $A$ is a cai of $C^*(A)$, the latter state cannot vanish on $A$.
  \end{proof}

{\bf Remark.}  Note that if ${\rm Re}(x) \geq \epsilon 1$ in $C^*(A)^1$, then there exists
a constant $C > 0$ with ${\rm Re}(x) \geq \epsilon 1 \geq C x^* x$,
and it follows that $x \in \Rdb_+ {\mathfrak F}_A$.  
Thus if $A$ is unital then every
strictly real positive in $A$ is 
 in $\Rdb_+ {\mathfrak F}_A$.   However this is false 
if $A$ is approximately unital (it is even easily seen to be 
false in the $C^*$-algebra  $A = c_0$).  Conversely, note that 
if $A$ is an approximately unital operator algebra with no 
r-ideals and no identity, then every nonzero element of 
$\Rdb_+ {\mathfrak F}_A$ is strictly real positive
by \cite{BRI} Theorem 4.1.

We also remark that it is tempting to define 
an element  $x \in A$ to be strictly real positive if Re$(x)$ strictly positive in some
$C^*$-algebra generated by $A$.  However this definition 
can depend on the particular generated $C^*$-algebra, unless one only uses
states on the latter that are not allowed to vanish on $A$ (in which
case it is equivalent to other definition).   As an example of this, consider the 
algebra of $2 \times 2$ matrices supported on the first row, and the various $C^*$-algebras it
can generate.

\medskip

We next discuss how results  in \cite{BRI}  generalize, particularly those 
related to strict real positivity if 
we use the definition at the start of the present section.
We recall that in \cite{BRI}, many `positivity' results were established for elements in ${\mathfrak F}_A$ or $\frac{1}{2} {\mathfrak F}_A$,
and by extension for the proper cone ${\mathfrak c}_A = 
\Rdb_+ {\mathfrak F}_A$.  In  \cite[Section 3]{BRII} we pointed out several of these facts
that generalized to the larger cone ${\mathfrak r}_A$, and indicated that some of this would be discussed in more detail in
\cite{BBS}.  In  \cite[Section 4]{BRII} we pointed out that the hypothesis in many of these results that 
$A$ be approximately unital could be simultaneously relaxed. 
In the next 
few paragraphs we give a few more details, that indicate the 
similarities and differences between these cones, particularly focusing on the results involving 
strictly real positive elements.    The following list should be added to the list in \cite[Section 3]{BRII},
and some complementary details are discussed in \cite{BBS}.  

In \cite[Lemma 2.9]{BRI} the ($\Leftarrow$) direction
is correct  for $x \in {\mathfrak r}_A$ with the same proof.   Also one need not assume there that 
$A$ is approximately unital, as we said towards the end of Section 4 in 
\cite{BRII}.    
The other direction is not true in general (not even in  $A = \ell^\infty_2$, see example in \cite{BBS}),
but there is a partial result, Lemma \ref{sect} below.

In \cite[Lemma 2.10]{BRI},
 (v) implies (iv) implies (iii) (or equivalently (i) or (ii)),  with ${\mathfrak r}_A$
in place of ${\mathfrak F}_A$, using the ${\mathfrak r}_A$ version above of the ($\Leftarrow$) direction of 
\cite[Lemma 2.9]{BRI}, and \cite[Theorem 3.2]{BRII} (which 
gives $s(x) = s({\mathfrak F}(x))$).
 However none of the other implications in that lemma are correct, even in $\ell^\infty_2$.

Proposition 2.11 and Theorem 2.19 of \cite{BRI} are correct in their ${\mathfrak r}_A$ variant, which should 
be phrased in terms of 
 strictly real positive elements in ${\mathfrak r}_A$ as defined above at the start of
the present section.  
Indeed this variant of Proposition 2.11 is true even for nonunital algebras 
if  in the proof we replace $C^*(A)$ by $A^1$.  
Theorem 2.19  of \cite{BRI} may be 
seen using the parts of \cite[Lemma 2.10]{BRI} which are true for ${\mathfrak r}_A$
in place of ${\mathfrak F}_A$, and \cite[Theorem 3.2]{BRII} (which 
gives $s(x) = s({\mathfrak F}(x))$).    Lemma 2.14 of \cite{BRI} 
is  clearly false even in $\Cdb$, however
it is true with essentially the same proof if the elements $x_k$ there
are strictly real positive elements, or more generally if they are in ${\mathfrak r}_A$ and their  numerical ranges in $A^1$ intersects the imaginary 
axis only possibly at $0$.  
 Also,  this 
does not effect the correctness of the important results that follow it in \cite[Section 2]{BRI}.  Indeed as  stated 
in \cite{BRII},  all descriptions of r-ideals and $\ell$-ideals and HSA's from \cite{BRI} are valid 
with ${\mathfrak r}_A$ in place of ${\mathfrak F}_A$, sometimes by using \cite[Corollaries 3.4 and 3.5]{BRII}).  We remark that Proposition 2.22 of  \cite{BRI} is clearly false
with ${\mathfrak F}_A$ replaced by ${\mathfrak r}_A$, even in $\Cdb$.
  
Similarly, in  \cite{BRI} Theorem 4.1, (c) implies (a) and (b) there with ${\mathfrak r}_A$
in place of ${\mathfrak F}_A$.   However the Volterra algebra \cite[Example 4.3]{BRI}
is an example where (a) in \cite{BRI} Theorem 4.1 holds but not (c) (note that 
the Volterra operator $V \in {\mathfrak r}_A$, but $V$ is not strictly real positive in 
$A$).   The results in Section 3 of  \cite{BRI}  were discussed in subsection
\ref{fgsect} and \cite{BOZ}.  It follows  
as in \cite{BRI} that if $x$ is a strictly real 
positive element (in our new sense above) in
a nonunital approximately unital operator algebra $A$, then $xA$ is never closed.
For if $xA$ is closed then by  the                                
${\mathfrak r}_A$ version of \cite[Lemma 2.10]{BRI} discussed above, we have
$xA = A$.  Now apply Corollary \ref{Aha4} to see that $A$ has a left identity
(which as we said in subsection \ref{fgsect} forces it to have an identity).

\begin{lemma} \label{sect}  In an operator algebra $A$, 
suppose that $x \in {\mathfrak r}_A$ and either $x$ is strictly real positive, or 
the numerical range $W(x)$ of $x$ in $A^1$ is contained in a sector 
$S_\psi$ of angle $\psi < \pi/2$ (see notation above Lemma {\rm \ref{roots}}).  If $\varphi$ is a state 
on $A$ or more generally on $A^1$, 
then $\varphi(s(x)) = 0$ iff 
$\varphi(x) = 0$.   \end{lemma}  \begin{proof}   The one direction is as in 
\cite[Lemma 2.9]{BRI}  as mentioned above.  The strictly real positive case of the 
other direction is obvious (but non-vacuous in the $A^1$ case).  In the remaining case, 
 write $\varphi = \langle \pi(\cdot) \xi , \xi \rangle$
for a unital $*$-representation  $\pi$ of $C^*(A^1)$ on a Hilbert space $H$,
and a unit vector $\xi \in H$.
Then $W(\pi(x))$ is contained in a sector of the same angle.
By Lemma 5.3 in Chapter IV of \cite{NF} we have $\Vert \pi(x) \xi \Vert^2
= \varphi(x^* x) = 0$.  As e.g.\ in the proof of 
\cite[Lemma 2.9]{BRI} this gives $\varphi(s(x)) = 0$.
  \end{proof}

\begin{corollary} \label{sx}  Let $x \in {\mathfrak r}_A$ for  an operator algebra $A$.  If  $\varphi(x^{\frac{1}{n}}) = 0$
for some $n \in \Ndb, n \geq 2$, and state  $\varphi$ on $A$,
then $\varphi(s(x)) = 0$ and $\varphi(x^{\frac{1}{m}}) = 0$
for all $m \in \Ndb$.  Thus if $\varphi(s(x)) \neq 0$ for a state  $\varphi$ on $A$, 
then ${\rm Re}(\varphi(x^{\frac{1}{n}})) > 0$ 
for all $n \in \Ndb, n \geq 2$.
   \end{corollary}  \begin{proof}
It is clear that $s(x) = s(x^{\frac{1}{m}})$ for all $m \in \Ndb$, by using for example the fact from
\cite[Section 3]{BRII} that $x^{\frac{1}{n}} \to s(x)$ weak*.  
Since the numerical range of $x^{\frac{1}{n}}$ 
in $A^1$ is contained in a sector
centered on the positive real axis of angle $< \pi$,
$\varphi(s(x)) = \varphi(s(x^{\frac{1}{n}})) = 0$ by Lemma \ref{sect}.
As we said above, this implies that 
$\varphi(x)  = 0$, and the same argument applies with
$x$ replaced by $x^{\frac{1}{m}}$ to give $\varphi(x^{\frac{1}{m}}) = 0$.   

The last statement follows from this, since ${\rm Re}(\varphi(x^{\frac{1}{n}})) > 0$
is equivalent to $\varphi(x^{\frac{1}{n}}) \neq 0$ if $n \geq 2$.  \end{proof}

{\bf Remark.}   Examining the proofs of the last three results show that they are 
valid if states on $A$ are replaced by  nonzero functionals that extend to states on 
$A^1$, or equivalently extend to  a $C^*$-algebra generated by $A^1$.

\begin{corollary} \label{sx3}  
In an operator algebra $A$,
if $x \in {\mathfrak r}_A$ and $x$ is strictly real positive,
then $x^{\frac{1}{n}}$ is strictly real positive for all $n \in \Ndb$. 
 \end{corollary}

\begin{proof}
If $x^{\frac{1}{n}}$ is not strictly real positive for some $n \geq 2$,
then $\varphi(x^{\frac{1}{n}}) = 0$ for some state 
$\varphi$ of $A^1$ which is nonzero on $A$.  Such a state extends to a state 
on $C^*(A^1)$.  By the last Remark,  $\varphi(x)  = 0$ by Corollary \ref{sx}, a contradiction.
\end{proof}

We recall that a {\em $\sigma$-compact} projection in $B^{**}$ 
for a $C^*$-algebra $B$,
is an open projection $p \in B^{**}$ which is 
the supremum (or weak* limit)
of an increasing sequence in $B_+$ \cite{PedS}.
It is well known from
$C^*$-algebra theory that 
this is equivalent to saying that $p$ is the support projection
of a closed right ideal  in $B$ which has a countable left cai;
and also 
 equivalent to saying that $p$ is the support projection of a strictly 
positive element in the hereditary subalgebra defined by $p$.    

\begin{lemma} \label{cssigc}  If $A$ is a closed subalgebra of a $C^*$-algebra $B$, and if $p$ is 
an open projection in $A^{**}$ then the following are equivalent:
\begin{itemize} \item [(i)] $p$ is the support projection of
a closed right ideal  in $A$
 with a countable left cai.
\item [(ii)]  $p$ is $\sigma$-compact  in $B^{**}$ in the sense above.
\item [(iii)]  $p$ is the support projection of
a closed right ideal  in $A$ of the form 
$\overline{xA}$ for some $x \in {\mathfrak r}_A$.  That is, $p = s(x)$ 
for some $x \in {\mathfrak r}_A$. 
\item [(iv)]  There is a
sequence $x_n \in {\mathfrak r}_A$ with $x_n = p x_n  \to p$ weak*. 
\item [(v)] $p$ is the support projection of a strictly real  
positive element $x$ of the hereditary subalgebra defined by $p$.
\end{itemize} 
If these hold then the sequence $(x_n)$
in {\rm (iv)} can be chosen to be increasing with respect to $\preccurlyeq$,
and they, and the element $x$ in {\rm (iii)} and {\rm (v)}, can be chosen to be
in $\frac{1}{2} {\mathfrak F}_A$
and nearly positive.
  \end{lemma}  \begin{proof} We know from the theory
in \cite{BRI,BRII} that  (i) and (iii) are equivalent,
and the element $x$ in {\rm (iii)} can be chosen to be
in $\frac{1}{2} {\mathfrak F}_A$ 
and nearly positive.  Indeed one direction is similar to the argument in
the paragraph after Corollary \ref{pr}.   That these imply
(iv) is similar, clearly $x_n = x^{\frac{1}{n}}$ has the desired properties
for $n$ large enough.    

(iv) $\Rightarrow$ (iii) \ If $x_n \in {\mathfrak r}_A$ with $x_n = p x_n  \to p$ weak*, then $p$ is the support projection of
the closed right ideal  $J = \{ a \in A :
p a = a \}$.  Indeed it is easy to see that $J^{\perp \perp} = p A^{**}$.  Note that $\overline{\sum_k \, x_n A}$ is a left ideal in $A$ and $J$,
but actually equals $J$ since its weak* closure contains $p$ and hence contains
$p A^{**} = J^{\perp \perp}$.   By \cite[Proposition 2.14]{BRI} and \cite[Corollary 3.5]{BRII},  $\overline{\sum_k \, x_n A}
 =  \overline{xA}$ for some $x \in {\mathfrak r}_A$.

(iii) $\Rightarrow$ (v) \ 
Suppose that $D$ is the hereditary subalgebra defined by $p$.
If $x$ is as in (iii) then $x \in D = \overline{xAx}$, and 
$$\overline{xD} \subset D \subset \overline{x^{\frac{1}{2}} x^{\frac{1}{2}}
Ax} \subset \overline{x^{\frac{1}{2}} D}  \subset \overline{xA D} 
\subset \overline{x D} .$$
So $D = \overline{xD}$, and 
by our version of \cite[Theorem 2.19]{BRI} discussed earlier in this section, $x$ is a strictly real positive element of $D$. 

(v)  $\Rightarrow$ (iii)  \ If $x \in {\mathfrak r}_D \subset {\mathfrak r}_A$ is a strictly real positive element of $D$
then by our version of \cite[Theorem 2.19]{BRI}  
  $p = w^*\lim_n \, x^{\frac{1}{n}} = s(x)$. 
 
Finally, the equivalence of (i) and (ii): Let $I =  p B^{**} \cap B$ and 
$J = p A^{**} \cap A$.  As we said,
(ii) is equivalent to 
$I$ having a countable left cai. 
 From \cite[Section 2]{BHN} we have $I = JB$.
So if $J$ has a countable left cai then so does $I$.
Similarly, any  left cai for $J$ is a left cai for $I$.  
If $I$ has a countable left cai $(f_n)$ choose elements $e_n$ from 
the left cai in the last line such that $\Vert e_n f_n - f_n \Vert <
2^{-n}$.  Then since 
$$e_n a = e_n (a - f_n a) + (e_n f_n - f_n)a + f_n a, \qquad a \in A , $$
it is clear that $J$ has a countable left cai.      
\end{proof}  

A similar result holds for left ideals or HSA's.

If $A$ is an operator algebra then
an open projection $p \in A^{**}$ will
be said to be {\em $\sigma$-compact} with respect to $A$ if
it satisfies the equivalent conditions in the previous result.
These projections, and the above lemma, will be used in our
`strict Urysohn lemma' in Section 4.

\section{Positivity in the Urysohn lemma and peak interpolation}  \label{pury}  

In our previous work \cite{BHN,BRI, BNII, BRII} we had two main settings for
 noncommutative Urysohn lemmata for a subalgebra $A$ of a $C^*$-algebra $B$.  In both
settings we have a compact projection $q \in A^{**}$, dominated
by an open projection $u$ in $B^{**}$, and we seek to find $a \in {\rm Ball}(A)$ with
$aq = q a = q$, and both $a \, u^{\perp}$ and $u^{\perp} \, a$ either small or zero.
In the first setting $u \in A^{**}$ too, whereas this is not required
in the second setting.  We now ask if in both settings one may also
have $a \in
\frac{1}{2} {\mathfrak F}_A$ and nearly positive (hence
`positive' in our new sense, and as close as
we like to a positive operator in the usual sense).
In the first setting, all works perfectly:

\begin{theorem} \label{urysII}  Let $A$ be an
operator algebra (not necessarily approximately unital), and let
$q \in A^{**}$ be a compact projection, which
is dominated by an open projection  $u \in A^{**}$.  Then
there exists nearly positive
 $a \in \frac{1}{2} {\mathfrak F}_A$ with $a q = qa = q$,
and $a u = ua = a$.   
\end{theorem}
 \begin{proof}  The proofs of \cite[Theorem 2.6]{BNII} and \cite[Theorem 6.6 (2)]{BRII}
show that this all can be done with $a \in \frac{1}{2} {\mathfrak F}_A$.
 Then $a^{\frac{1}{n}} q = q a^{\frac{1}{n}}  = q$, as is clear for example
using the power series form $a^{\frac{1}{n}} = \sum_{k=0}^\infty
\, { 1/n \choose k} \, (-1)^k (1-a)^k$ from \cite[Section 2]{BRI},
where it is also shown that $a^{\frac{1}{n}} \in \frac{1}{2} {\mathfrak F}_A$.
Similarly $a^{\frac{1}{n}} u = u a^{\frac{1}{n}} = a^{\frac{1}{n}}$, since $u$ is
the identity multiplier on oa$(a)$, and
oa$(a)$ contains these roots \cite[Section 2]{BRI}.
That the numerical range of $a^{\frac{1}{n}}$ lies in
a  cigar centered on the line segment $[0,1]$ in the
$x$-axis, of height $< \epsilon$
is as in the proof of \cite[Theorem 2.4]{BRI}. 
\end{proof}

We now turn to the second
setting  (see e.g.\ \cite[Theorem 6.6 (1)]{BRII}), where the dominating open projection
$u$ is not required
to be in $A^{\perp \perp}$.
Of course if $A$ has no identity or cai then one cannot expect
the `interpolating' element $a$ to be in $\frac{1}{2} {\mathfrak F}_A$ or
${\mathfrak r}_A$.  This  may  be seen clearly in the case that $A$ is the functions in the disk algebra vanishing
at $0$.  Here $\frac{1}{2} {\mathfrak F}_A$ and 
${\mathfrak r}_A$ are $(0)$.  Indeed by the maximum modulus theorem for harmonic functions there are  
no nonconstant 
functions in this algebra which have nonnegative real part.
The remaining question is the approximately unital case `with positivity'.
We solve this next, also solving the questions posed at the end of \cite{BNII}.

\begin{theorem} \label{urys}  Let $A$ be an approximately unital
subalgebra of a  $C^*$-algebra $B$, and let
$q \in A^{\perp \perp}$ be a compact projection.
\begin{itemize}
\item [(1)]  If $q$ is dominated by an open projection  $u \in B^{**}$
then for any $\epsilon > 0$,
there exists an $a \in 
\frac{1}{2} {\mathfrak F}_A$ with $a q = qa =  q$,
and $\Vert a (1-u) \Vert < \epsilon$ and $\Vert (1-u)  a \Vert < \epsilon$.
Indeed this can be done with in addition $a$ nearly positive
(thus the numerical range (and spectrum) of $a$ within
a horizontal cigar centered on the line segment $[0,1]$ 
in the
$x$-axis, of height $< \epsilon$).  
\item [(2)]  $q$ is a weak* limit of a net $(y_t)$ of nearly
positive elements in $\frac{1}{2}
\, {\mathfrak F}_A$ 
with $y_t q = q y_t = q$.
\end{itemize} \end{theorem}
 \begin{proof}    (2) \  First assume that $q = u(x)$ (this 
was defined in the Introduction)
for some $x \in \frac{1}{2} {\mathfrak F}_A$.
We may replace $A$ by the commutative algebra ${\rm oa}(x)$, and then $q$ is a minimal projection,
 since $q \, p(x) \in \Cdb q$ for any polynomial $p$.  
Now $q$ is closed and compact in $(A^{1})^{**}$, so by the unital case of (2), which follows from 
\cite[Theorem 2.24]{BRI} and the closing remarks to \cite{BNII}, there is a net 
$(z_t) \in \frac{1}{2}
\, {\mathfrak F}_{A^1}$ with $z_t q = q z_t = q$ and $z_t \to q$ weak*.  Let 
$y_t  = z_t^{\frac{1}{2}} x^{\frac{1}{2}}$.  By  
 \cite[Lemma 4.2 (3)]{BBS}, 
 we have $y_t 
\in \frac{1}{2} {\mathfrak F}_{A^1}
\cap A = \frac{1}{2} {\mathfrak F}_{A}$.  Also, $x^{\frac{1}{2}} q = q x^{\frac{1}{2}} = q$  by considerations
used in the last proof, and similarly $z_t^{\frac{1}{2}} q = q z_t^{\frac{1}{2}} = q$.
Thus $y_t^{\frac{1}{2}} q = q y_t^{\frac{1}{2}} = q$.   If $A$ is represented nondegenerately on a Hilbert
space $H$, and we identify $1_{A^1}$ with $I_H$, then
 for any $\zeta \in H$ we have by a result at the end of the Introduction  
that 
$$\Vert (y_t - q) \zeta \Vert
= \Vert (z_t^{\frac{1}{2}} - q) x^{\frac{1}{2}} \zeta \Vert
\leq K \Vert (z_t - q) x^{\frac{1}{2}} \zeta \Vert^{\frac{1}{2}} \to 0 .$$
Thus $y_t \to q$ strongly and hence weak*. 

Next, for an arbitrary compact projection $q \in  A^{\perp \perp}$, by \cite[Theorem 3.4]{BNII}
there exists a net $x_s \in \frac{1}{2} {\mathfrak F}_A$  with $u(x_s) \searrow q$.
By the last paragraph there exist nets $y_t^s \in \frac{1}{2} \, {\mathfrak F}_A$
with $y_t^s \, u(x_s) = u(x_s) \, y_t^s = u(x_s)$, and $y_t^s \to u(x_s)$ weak*.
Then $$y_t^s q = y_t^s \, u(x_s) \, q = u(x_s) \, q = q,$$ for each $t, s$.  It is clear that
the $y_t^s$ can be arranged into a net weak* convergent to $q$.

(1) \ If $A$ is unital then the first assertion of (1) 
is \cite[Theorem 2.24]{BRI}.   In the approximately unital case,
by the ideas in the closing remarks to \cite{BNII},
the first assertion of (1) should be equivalent to (2).   Indeed, by 
substituting such a net $(y_t)$ into
the proof of \cite[Theorem 2.1]{BNII} one obtains the
first assertion of (1).

Finally, we obtain the `cigar' assertion.
For $(y_t)$ as in (2), similarly to the last paragraph
we substitute the net $(y_t^{\frac{1}{m}})$ into
the proof of \cite[Theorem 2.1]{BNII}.  Here $m$ is a fixed integer so large that the
numerical range of $y_t^{\frac{1}{m}}$ lies within the appropriate  horizontal cigar.
As in the proof of the previous theorem, 
$y_t^{\frac{1}{m}} q = q y_t^{\frac{1}{m}} = q$ and
$y_t^{\frac{1}{m}} \to q$ weak* with $t$ since if $\zeta \in H$ again 
then  
$$\Vert (y_t^{\frac{1}{m}} - q) \zeta \Vert
\leq \Vert (y_t - q) \zeta \Vert^{\frac{1}{m}} \to 0$$ 
by the inequality 
at the end of the Introduction. 
  \end{proof}

{\bf Remark.}    
The  recent paper  \cite{CGK} 
contains a special kind of  `Urysohn lemma with positivity' for function algebras.  It seems that our Urysohn lemma applied to a function algebra 
has much weaker (fewer) hypotheses, and has stronger conclusions 
except that our interpolating element has range in the usual thin
cigar in the right half plane which we like to use, and this is 
contained in their Stolz region which 
contains $0$ as an interior point, except for a
tiny region just to the left of $1$.  
Hopefully our results  could be helpful
in such applications.

\bigskip

We now turn to our analogue of the `strict Urysohn lemma'.   We recall that the classical form of the
strict Urysohn lemma in topology finds a positive continuous function  which is $0$ and $1$ on the two
given closed sets, and which is   strictly between $0$ and $1$ outside of these two sets.   The latter is
essentially equivalent to saying that there is a positive contraction $f$ in the algebra 
such that the  given closed sets are peak sets for $f$ and $1-f$.   With this in mind we 
 state some preliminaries related 
to peak projections, the noncommutative generalization of peak sets. There are many equivalent definitions of peak projections (see e.g.\ \cite{Hay,BHN,BNII,BRII}), but basically  they are the 
 closed projections $q$ for which there is a contraction $x$ with $xq = q x = q$ (so $x$ `equals one on' $q$), 
and $|x| < 1$ in some sense (which is made precise in the above references) 
on $q^\perp$; we write $q = u(x)$.   More generally, if $B$ is a $C^*$-algebra and $x \in {\rm Ball}(B)$ one can define
$u(x) = w^*\lim_n \, x (x^* x)^n$, which always exists in $B^{**}$ and is a partial isometry
(see e.g.\ \cite{ERut} and \cite[Lemma 3.1]{BNII}).   When this is a nonzero projection it is
a peak projection and equals $w^*\lim_n \, x^n$, and we say that $x$ {\em peaks at} this projection.  This is the case for example for any norm $1$  element 
of $\frac{1}{2} {\mathfrak F}_A$ (see \cite[Corollary 3.3]{BNII}), for any closed subalgebra  $A$ of $B$, 
and here the peak projection $u(x)$ is in$ A^{**}$.  However it need not be the case for any norm $1$  real positive element, 
even in a unital $C^*$-algebra.  For example if  $V$ is the Volterra 
operator, which is accretive, and $x = \frac{1}{\Vert V + \epsilon I \Vert} (V + \epsilon I)$, then one can show that 
 $w^*\lim_n \, x^n = 0$.     The following is 
implicit in \cite[Lemma 3.1]{BNII}: 

\begin{lemma} \label{chux}   Suppose that $B$ 
is a $C^*$-algebra, that $x \in {\rm Ball}(B)$ and that $q \in B^{**}$ is a closed projection with $qx = q$.
Then $x$ peaks at $q$ (that is, $q$ is a peak projection and equals $u(x)$) iff $\varphi(x^* x) < 1$ for every state $\varphi$ of $B$ with $\varphi(q) = 0$.
\end{lemma}

\begin{proof}  ($\Leftarrow$) \ This follows from \cite[Lemma 3.1]{BNII}.

 ($\Rightarrow$) \  If $q = u(x)$ then the last assertions of  \cite[Lemma 3.1]{BNII} show that (3) there holds. If 
$\varphi$ is a state of $B$ with $\varphi(q) = 0$, then $\varphi(1-q) = 1$ and by Cauchy-Schwarz
$\varphi(x^* xq)  = 0$,   So $\varphi(x^* x) < 1$ by (3) there.  \end{proof}

{\bf Remark.}  In place of using states  with $\varphi(q) = 0$ in the lemma and its application below,
one can use minimal or  compact projections dominated by $1-q$, as in the proof of
\cite[Theorem 3.4 (2)]{BNII}.

\begin{lemma} \label{sigcp}   Suppose that $A$ 
is an approximately
unital operator algebra, that $q \in A^{**}$ is compact,
and that $p = e-q$ is $\sigma$-compact in $A^{**}$.  Then $q$ is a 
peak projection for $A$, indeed $q = u(x)$ for some
nearly positive $x \in \frac{1}{2} {\mathfrak F}_A$.      
\end{lemma}

\begin{proof}  It is only necessary to find such $x \in \frac{1}{2} {\mathfrak F}_A$, the claim about near positivity will follow
from \cite[Corollary 3.3]{BNII}.   If $A$ is unital then by \cite[Proposition 2.22]{BRI} we have $q = s(a)^\perp = u(1-a)$, and we 
are done.  If $A$ is nonunital by the above applied in $A^1$ we have $1-s(a) = u(b)$ where $b = 1-a \in \frac{1}{2} {\mathfrak F}_{A^1}$.
Since $q$ is compact there exists $r \in {\rm Ball}(A)$ with $q = rq$.   
We follow the idea in the proof of \cite[Theorem 3.4 (3)]{BNII}.  Let $d = rb
\in {\rm Ball}(A)$, then $$dq = rbq = rb (1-p)q = r (1-p)q = rq = q .$$
If $\varphi \in S(B)$ with $\varphi(q) = 0$,   then $\varphi$ extends to a state $\psi \in S(B^1)$ with $\psi(q) = 0$.
By Lemma \ref{chux} applied in $B^1$ we have $\psi(b^* b) < 1$, so that 
$$\varphi(d^* d) = \psi(d^* d) < \psi(b^* b) < 1.$$  Thus $q= u(d)$  is a 
peak projection for $A$ by \cite[Corollary 3.3]{BNII}.  By  \cite[Theorem 3.4 (3)]{BNII}, 
$q = u(x)$ for some
 $x \in \frac{1}{2} {\mathfrak F}_A$.  
\end{proof}

\begin{corollary}  \label{ux}  Suppose that $A$ is a (not necessarily
approximately unital) operator algebra, and $B$ is a $C^*$-algebra containing 
$A$.
If a peak projection for $B$ lies in $A^{\perp \perp}$ then it is also a
peak projection for $A$.
\end{corollary}

\begin{proof}     Suppose that $q = u(b) \in  A^{\perp \perp} \subset (A^1)^{\perp \perp}$ for some
$b \in  \frac{1}{2} {\mathfrak F}_B$.    Then by \cite[Proposition 2.22]{BRI} we have 
$s(1-b) = 1-q$ is a $\sigma$-compact projection in $(A^1)^{\perp \perp}$.   So by Lemma \ref{cssigc} 
we have that $1-q = s(a)$ for some $a \in \frac{1}{2} {\mathfrak F}_{A^1}$.  By \cite[Proposition 2.22]{BRI} 
again, $q = u(1-a)$.  By  \cite[Proposition 6.4]{BRI}, $q$ is  a
peak projection for $A$. \end{proof}  

\begin{theorem} \label{sul}  {\rm (A strict noncommutative Urysohn lemma for operator algebras)} \  Suppose that $A$ is any (possibly not approximately unital) operator algebra and that 
$q$ and $p$ are respectively compact and open projections in $A^{**}$ with $q \leq p$, and $p - q$ $\sigma$-compact. 
Then there exists $x \in  \frac{1}{2} {\mathfrak F}_{A}$ such that $xq = qx = q$ and $x p = px = x$,
and such that $x$ peaks at $q$ (that is, $u(x) = q$) and $s(x) = p$, and $1-x$ peaks at $1-p$ with respect to $A^1$
(that is, $u(1-x) = 1-p$).   The latter identities imply
 that $x$ is real strictly positive
in the hereditary subalgebra $C$ associated with $p$, and $1-x$  is real strictly positive
in the hereditary subalgebra  in $A^1$ associated with $1-q$.   Also, $s(x(1-x)) = p-q$,
so that $x(1-x)$ is  real strictly positive
in the hereditary subalgebra  in $A$ associated with $p-q$.  We can also have $x$  `almost positive',  in the sense that 
if $\epsilon > 0$ is given one can choose $x$ as above but also  satisfying 
${\rm Re} (x) \geq 0$ and $\Vert x - {\rm Re} (x) \Vert < \epsilon$.  
 \end{theorem}  \begin{proof}   Consider the hereditary subalgebra $C$ associated with $p$.   It is clear e.g.\ from
Lemma \ref{cssigc}, that  $p-q$ is a $\sigma$-compact projection with respect to  $C$.  
Applying Lemma \ref{sigcp} in $C$, we can choose 
$b \in \frac{1}{2} {\mathfrak F}_C \subset  \frac{1}{2} {\mathfrak F}_A$ with
$u(b) = q$.  
By the last Urysohn lemma above, we can choose 
$r  \in \frac{1}{2} {\mathfrak F}_A$ with $r p =  p r = r$ and $r q = q r = q$.
The argument in the 
proof of \cite[Theorem 3.4 (3)]{BNII} shows that the  closed algebra
$D$ generated by 
$x = r b r$ and $b$, is approximately unital, and that there is
an element $f_2 \in D \cap \frac{1}{2} {\mathfrak F}_A$ with $u(f_2) = q$.
Note that $f_2 p = p f_2= f_2$.    By taking roots we can assume that $f_2$ is nearly positive.

Similarly, but working in $A^1$,
one sees that there is a nearly positive $f_1 \in \frac{1}{2} {\mathfrak F}_{A^1}$ with $f_1 q=  q f_1 = 0$ and $u(f_1) = 1-p$.  We have $f_1 (1-p) = 
1-p$ which implies that $(1-f_1) p = 1-f_1$.
Let $x =  \frac{1}{2} (f_2 + (1-f_1)) \in \frac{1}{2} {\mathfrak F}_{A^1}$. 
Since $f_1, f_2$ are nearly positive, it is easy to see that $x$ is almost positive in the sense above, by a variant of the 
computation involving $\Vert {\rm Im} \, (x) \Vert$ in one of the early paragraphs of our paper.   
We have $1-x = \frac{1}{2} ((1-f_2) + f_1)$. 
Within $(A^{1})^{**}$ we have by \cite[Proposition 1.1]{BNII} (and the fact that a tripotent dominated by a projection in the natural
ordering on tripotents
is
a projection) that $$u(x) = u(f_2) \wedge u(1-f_1) = u(f_2) = q,$$
since $(1-f_1) q= q$ which implies $u(1-f_1) \geq q$.   Similarly $$u(1-x) = u(1-f_2)
\wedge u(f_1) =    u(f_1) =    1-p,$$ since $(1-f_2) (1-p) = 1-p$ and so $u(1-f_2) \geq 1-p$.  

Since $x p  = x$, and $p \in A^{\perp \perp}$, and $A^{\perp \perp}$ is an ideal in 
$(A^{1})^{**}$,  we see that
$x \in A^{\perp \perp} \cap A^1 = A$.    

Note that $\Vert 1 - 4 (x - x^2) \Vert = \Vert (1 - 2 x)^2 \Vert  \leq 1$, so $x (1-x) \in \frac{1}{4}  {\mathfrak F}_{A}$.
Then by Lemma \ref{rootf0}, 
$s(x(1-x))$ may be regarded as the strong limit 
of $(x(1-x))^{\frac{1}{n}} = x^{\frac{1}{n}} (1-x)^{\frac{1}{n}}$ (see e.g.\ \cite{BBS} for the last identity), which is $s(x) s(1-x) = p(1-q) = p-q$.
The `strictly real positive' assertions follow from Lemma  \ref{cssigc}.
\end{proof}  

{\bf Remarks.}    1) \ One may replace the hypothesis in Theorem \ref{sul} that $p - q$ be $\sigma$-compact, by the premise that 
both $q$ and $1-p$ are peak projections (in $A$ and $A^1$).  Indeed if $q = u(w), 1-p = u(z)$, then by \cite[Corollary 3.5]{BNII} 
$1 + q - p = u(w) + u(z) = u(k)$ say.  Hence by \cite[Proposition 2.22]{BRI} and 
Lemma \ref{cssigc}  we have that $p - q = 1 - u(k) = s(k)$ is $\sigma$-compact.  

\smallskip

2) \ Under a commuting hypothesis we offer a quicker proof inspired by the proof of
\cite[Theorem 2]{PedS}: choose $b \in \frac{1}{2} {\mathfrak F}_A$ with $s(b) = p-q$.  
Then if $r$ is as in the last proof, and $br = rb$, set $x = (1-r) b + (1-b) r$.  Then 
$1 - 2x = (1-2b)(1-2r)$, a contraction, so that $x \in \frac{1}{2} {\mathfrak F}_A$, and it is
easy to see that $x q = q$ and $x p = x$.  

\bigskip

We give an application of our strict noncommutative Urysohn lemma  to the lifting of projections, a variant of  
\cite[Corollary 4]{PedS}.   First we will need a sharpening of \cite[Proposition 6.2]{BRI}.  Recall that if $A$ is an 
operator algebra containing a closed approximately unital two-sided ideal $J$ with 
 support
projection $p$, then  $p$ is central in 
$(A^1)^{**}$ since $J$ is a two-sided ideal.  We may view $A/J \subset A^{**}
(1-p)$ via the map $a+J \mapsto a (1-p)$, in view of the identifications
$$(A/J)^{**} \cong A^{**} /J^{\perp \perp} \cong A^{**} /A^{**} p \cong A^{**}
(1-p) \subset A^{**}.$$  

\begin{lemma} \label{broy}    Let $A$ be an operator algebra containing a closed approximately unital two-sided ideal $J$ with 
 support
projection $p$, and suppose that $D$ is a HSA in $A/J$.  Regarding
$(A/J)^{**} \cong  A^{**}
(1-p)$ as above, let $r$ be the projection in $A^{**}
(1-p)$ corresponding to the support projection of $D$.  Then 
the preimage of $D$ in $A$ under the quotient map is a HSA in $A$ with support projection $p + r$.
\end{lemma}

\begin{proof}   By the proof of \cite[Proposition 6.2 and Corollary 6.3]{BRI}, the preimage $C$ of $D$ in $A$ is a HSA in $A$, 
and $C/J \cong D$.  Thus $C^{**} \cong J^{\perp \perp} \oplus^\infty D^{**}$, and we can view 
 the isomorphism $C^{**} \to J^{\perp \perp} \oplus^\infty D^{**}$ here as the restriction of the completely isometric
map $\eta \mapsto (\eta p, \eta(1-p))$ setting up the isomorphism  $A^{**} \cong J^{\perp \perp} \oplus^\infty A^{**}
(1-p)$.
If $\eta \in A^{**}$ with $\eta(1-p) \in r A^{**} r  \cong D^{**}$, then $\eta \in  \eta p + r A^{**} r \subset (p+r) A^{**} (p+r)$.
Hence $$C^{\perp \perp} = 
\{ \eta \in A^{**} : \eta(1-p) \in r A^{**} r  \} = (p+r) A^{**} (p+r).$$   Thus $p+r$ is the support
projection of $C$, so is open. \end{proof}  

\begin{corollary} \label{liftp}   Let $A$ be an 
operator algebra containing a closed approximately unital two-sided ideal $J$ with $\sigma$-compact support
projection $p$, and suppose that $q$ is a projection in $A/J$.  Then there exists 
almost positive $x \in \frac{1}{2} {\mathfrak F}_A$ such that $x + J = p$ and $s(x(1-x)) = p$,
so that $x(1-x)$ is  real strictly positive
in  $J$.   Also, the peak 
$u(x)$ for $x$ equals the canonical copy
of $q$ in $A^{**} (1-p)$.  \end{corollary}   \begin{proof} 
 By  \cite[Proposition 2.22]{BRI}, $q$ has a lift $y \in \frac{1}{2} {\mathfrak F}_A$, so that the copy of $q$ in 
$A^{**} (1-p)$ is $r = y (1-p)$.  Thus  $r$ is a projection in $A^{**}$.   Also, 
$r = (y (1-p))^n = y^n (1-p) \to u(y) (1-p)$ weak*.   This implies that $r = u(y) (1-p) = 
u(y) \wedge (1-p)$ is a closed projection in $A^1$,
hence is a compact projection in $A^{**}$.   Clearly $r 
= y r$.      By Lemma \ref{broy}  
the projection $p + r$ is open in $A^{**}$, and it dominates $r$.   We apply Theorem \ref{sul} to see that 
there exists almost positive  $x \in  \frac{1}{2} {\mathfrak F}_{A}$ such that $u(x) = r$, 
and $xr = rx = r$ and $x (p +r) = (p + r)x = x$.  Thus $y (1-p) = r = xr = x(1-p)$, and so 
 $x + J = q$.     Also, $s(x(1-x)) = p$.   
\end{proof}

We now turn to noncommutative peak interpolation.     The following is an improvement of \cite[Lemma  2.1]{Bnpi}.
 
\begin{proposition} \label{interpnp}  Suppose that $A$ is an approximately unital operator algebra, and $B$ is a $C^*$-algebra generated by $A$.  If
$c \in B_+$ with $\Vert c \Vert < 1$ then there
exists an $a \in \frac{1}{2} {\mathfrak F}_A$ with $|1 - a|^2 \leq 1-c$.
Indeed such $a$ can be chosen to also be nearly positive.  
\end{proposition} \begin{proof}   By Theorem \ref{havin} (2'), there exists 
nearly positive $a \in \frac{1}{2} {\mathfrak F}_A$ with $$c \leq {\rm Re}(a) \leq 2 {\rm Re}(a) - a^* a$$
since  $a^* a \leq {\rm Re}(a)$ if $a \in \frac{1}{2} {\mathfrak F}_A$.
Thus $|1-a|^2 \leq 1-c$.
 \end{proof}   

In the last result one cannot hope to replace the hypothesis 
$\Vert c \Vert < 1$ by $\Vert c \Vert \leq 1$, as can be seen 
with the example in Remark 1 after Theorem \ref{havin}.

Proposition \ref{interpnp},
like several other results in this
paper,  is equivalent to Read's theorem from \cite{Read}.
Indeed if $e$ is an identity of norm $1$ for $A^{**}$,
and if we choose $a_t
\in \frac{1}{2} {\mathfrak F}_A$ with
$|e - a_t|^2 \leq e - e_t$, where $(e_t)$ is any positive cai in ${\mathcal U}_B$,
then we have
$$|\langle (e - a_t) \zeta , \eta \rangle|^2 \leq 
\Vert (e - a_t) \zeta \Vert^2 = \langle |e - a_t|^2 \zeta, \zeta \rangle
\leq \langle (e - e_t) \zeta, \zeta \rangle \to 0,$$
for all $\zeta \in H$.   Thus $e$ is a weak* limit of a net in
$\frac{1}{2} {\mathfrak F}_A$, and hence by the usual argument
there exists a cai in $\frac{1}{2} {\mathfrak F}_A$.
  
 As in \cite[Lemma 2.1]{Bnpi}, Proposition \ref{interpnp} can be interpreted as 
a noncommutative peak interpolation result.  Namely, if the projection 
$q = 1_{A^1}-e$ is dominated by $d = 1-c$ then there exists an element $g = 1-a 
\in A^1$ with $g q = q g = q$, and $g^* g \leq d$.    The new point is
that $a$ is in $\frac{1}{2} {\mathfrak F}_A$ and nearly positive.

This leads one to ask whether the other noncommutative peak interpolation results
we have obtained in earlier papers can also be done with the interpolating element in $\frac{1}{2} {\mathfrak F}_A$,
or more generally with the interpolating element having prescribed numerical range.   
We will discuss this below.   As discussed at the end of \cite{BOZ}, lifting 
elements without increasing the norm while keeping the numerical range in a fixed compact convex subset $E$ of the plane, may be regarded
as a kind of Tietze extension theorem. (In the usual Tietze theorem $E = [-1,1]$.    It should be 
pointed out that in the usual Tietze theorem one can lift 
elements from the multiplier algebra whereas here we are being more modest.)   We refer the reader to \cite[Section 3]{Brown} for a discussion of some other kinds of Tietze theorems for $C^*$-algebras.

The following two theorems may be regarded as peak interpolation theorems `with positivity'.   
They are generalizations of \cite[Theorem 5.1]{BRII} (see also Corollary 2.2 in that reference). 

\begin{theorem} \label{peakthang1}   Suppose that $A$ is an operator algebra
(not necessarily approximately unital),
and that  $q$ is a closed projection in $(A^1)^{**}$.
Suppose that $b \in A$ with $b q= qb$ and $\Vert b q \Vert \leq 1$,
and $\Vert (1-2b) q \Vert \leq 1$.  Then  there exists an element $g \in  \frac{1}{2} {\mathfrak F}_A
\subset {\rm Ball}(A)$ with $g q =  q g = b q$.
\end{theorem}

\begin{proof}  We modify the proof of \cite[Theorem 5.1]{BRII}.  In that 
proof a closed subalgebra $C$ of $A^1$ is constructed which 
contains $b$ and $1_{A^1}$, such that $q$ is in the center
of $C^{\perp \perp} \cong C^{**}$.    So $q^\perp$ supports a closed 
two-sided ideal $J$ in $C$.  
Then we set $I = C \cap A$, an ideal in $C$ containing $b$.
Finally, an $M$-ideal $D$ in $I$ was contructed there; this will be an 
approximately unital ideal in $I$.   Using the language 
of the proof of \cite[Theorem 5.1]{BRII},
since $P(I^{\perp}) \subset I^{\perp}$ it follows
that $I^{\perp \perp}$ is invariant under $P^*$.
By \cite[Proposition I.1.16]{HWW} we have that 
$I + \tilde{D}$ is closed, hence it 
follows similarly to the centered equation in the proof of \cite[Proposition 7.3]{BRI}, and the
two lines above it,  that 
$$D^{\perp \perp} = (I \cap \tilde{D})^{\perp \perp} =
I^{\perp \perp} \cap  \tilde{D}^{\perp \perp} =  
(1-q) C^{**} \cap I^{\perp \perp} = (1-q) I^{\perp \perp}.$$ 
Thus the $M$-projection from $I^{**}$ onto $D^{\perp \perp}$ is
multiplication by $1-q$, which is also 
the restriction of $P^*$ to $I^{\perp \perp}$.
Now $I/D$ is an operator algebra;
indeed it may be viewed, via the map $x + D \mapsto q x$,
as a subalgebra of $$(I/D)^{**} \cong
I^{**}/D^{\perp \perp} \cong q I^{\perp \perp} \subset q C^{**} \subset
q (A^1)^{**} q.$$   
Indeed it is not hard to see that $I/D$ may be regarded as an ideal in the unital subalgebra $C/J$
of $q C^{**}$, where $J$ was defined above.   

If  $\Vert (1-2b) q \Vert \leq 1$ then $bq
\in \frac{1}{2} {\mathfrak F}_{q(A^1)^{**} q}$, so that  
$b +D \in \frac{1}{2} {\mathfrak F}_{I/D}$.   Hence by \cite[Proposition
6.1]{BRI} there
exists $g \in \frac{1}{2} {\mathfrak F}_I
\subset \frac{1}{2} {\mathfrak F}_A$ with $g + D = b+D$.   
We have $g q = q g = bq$.  
\end{proof}

We will need a  
simple corollary of Meyer's theorem mentioned in the introduction:

\begin{lemma} \label{Meyer}   Suppose that $A$ and $B$ are closed subalgebras 
of unital operator algebras $C$ and $D$ respectively, with $1_C \notin A$ and $1_B \notin D$,
  and that $q : A \to B$ is a complete quotient map and homomorphism.  Then the unique unital extension 
of $q$ to a unital map from $A + \Cdb 1_C$ to $B + \Cdb 1_D$, is a complete quotient map.
\end{lemma}

\begin{proof}  Let  $J = {\rm Ker} \, q$, let $\tilde{q} : A/J \to B$ 
be the induced complete isometry, and let 
$\theta : A + \Cdb 1_C \to B + \Cdb 1_D$ be the unique unital extension of $q$.    
This gives a one-to-one homomorphism $\tilde{\theta} : (A + \Cdb 1_C)/J \to B + \Cdb 1_D$
which equals $\tilde{q}$ on $A/J$.  If $B$, and hence $A/J$, is not unital then $\tilde{\theta}$ is 
a completely isometric isomorphism by Meyer's result mentioned in the Introduction
(since both $(A + \Cdb 1_C)/J$ and $B + \Cdb 1_D$ are `unitizations' of 
$A/J \cong B$).   Similarly, if $B$ is unital, then $\tilde{\theta}$ is 
a completely isometric isomorphism by the (almost trivial) 
uniquess of the unitization
of an already unital operator algebra.  So in either case we may deduce that
$\tilde{\theta}$ is a complete isometry and $\theta$ is a complete quotient map.  \end{proof}

The following is a noncommutative peak interpolation theorem which is
also, as discussed in the paragraph before Theorem  \ref{peakthang1},   a  kind of `Tietze theorem'.
 It also yields
a peak interpolation theorems `with positivity': 
if one insists that the set $E$ appearing here lies in the right half plane, or the usual `cigar' centered on $[0,1]$,
then the interpolation or extension is preserving `positivity' in our new sense.

\begin{theorem} \label{peakthang2}  {\rm (A noncommutative Tietze theorem)} \
 Suppose that $A$ is an operator algebra
(not necessarily approximately unital),
and that  $q$ is a closed projection in $(A^1)^{**}$.
Suppose that $b \in A$ with $b q= qb$ and $\Vert b q \Vert \leq 1$, and that  the numerical range of 
$bq$ (in e.g.\ $q (A^1)^{**}q$) 
is contained in a compact convex set $E$ in the plane.  We also suppose, by
fattening it slightly if necessary, that $E$ is not a line segment.   If both $A$ is nonunital and if $q \in A^{\perp \perp}$,  
then we will also insist that $0 \in E$.     Then there exists an element $g \in {\rm Ball}(A)$ with $g q =  q g = b q$, 
such that  the numerical range of 
$b$ with respect to  $A^1$ is contained in $E$.
\end{theorem}

\begin{proof}   This is the same as the last proof except that the last paragraph should 
be replaced by the following.  Suppose that the numerical range $W_{qC^{**}}(bq)$ lies in the convex set $E$ described.
If $1_{A^1} \in I$ (which is the case for example if $A$ is unital) then $I/D$ viewed  in $qC^{**}$ as above has identity $q$.   Then the 
numerical range of $b + D$ in $I/D$ 
 is a subset of $E$.  
By \cite[Theorem 3.1]{CSSW}
and the Claim at the end of \cite{BOZ}, there
exists a contractive lift $g \in I \subset C$ with numerical range with respect to $C$,
and hence with respect to $A^1$, contained in $E$.    We have  $g q = q g = bq$  since 
 $g + D = b+D$.    This proves the result.  Thus henceforth we can assume that
$1_C = 1_{A^1} \notin I$ and that $A$ is nonunital.  

Next 
suppose that the copy $qI$ of $I/D$ in $qC^{**}$ above does not contain $q$.  This will be the case 
for example if  $q \notin A^{\perp \perp}$ (for if 
$q = qx$ for some $x \in I$ then $q \in q A \subset A^{\perp \perp}$, since the latter
is an ideal in $(A^1)^{**}$). 
 By Lemma \ref{Meyer} we can extend the 
quotient map $I \to I/D$ to a complete quotient map 
$\theta : I + \Cdb 1_C \to  I/D + \Cdb q$ (the latter viewed as above in $qC^{**}$).  
By \cite[Theorem 3.1]{CSSW}
and the Claim at the end of \cite{BOZ}, there
exists a contractive lift $g \in I + \Cdb 1_C$ with numerical range with respect to $C$,
and hence with respect to $A^1$, contained in $E$.   If $g = x + \lambda 1_C$ with $x \in I$
then
$b + D = g + D = \lambda q + x + D \in I/D$, which forces $\lambda = 0$.
So  $g \in I \subset A$,
and $g q = q g = bq$ again as above.  
 Finally, suppose that $I/D$ contains $q$, and $0 \in E$, so that $E = [0,1]E$.
Here $q$ is the identity of $C/J$ (viewed  as above in $qC^{**}$).
Since $I/D$ is an ideal, we have $I/D = C/J$.   
Consider $I \oplus c_0$ 
and its ideal $D \oplus (0)$. 
The quotient here is $(I/D) \oplus c_0$, which may be viewed 
as a subalgebra of $q C^{**} \oplus c$.   
The numerical range  of an element $(x,0)$ in a direct sum
$A_1 \oplus^\infty A_2$ of unital Banach algebras is easily seen to be 
$[0,1] W_{A_1}(x)$.   Hence the numerical range  of 
$(b+I,0)$ in $(I/D) \oplus c$ is contained in $[0,1]E =E$.   By Lemma \ref{Meyer} 
the canonical complete quotient map $I \oplus c_0 \to (I/D) \oplus c_0$ 
extends to a unital
 complete quotient map $(I \oplus c_0) + \Cdb (1_C,\vec 1) \to 
(I/D) \oplus c$.    
By  \cite[Theorem 3.1]{CSSW}
and the Claim at the end of \cite{BOZ}, there
exists a contractive lift  $(g,0) \in (I \oplus c_0) + \Cdb (1_C,\vec 1)$ whose
numerical range in the latter space, and hence
in $C \oplus c$,  is contained in $E$.   By the Banach algebra sum fact a few lines earlier,
we deduce that $W_C(g) \subset  E$, and hence $W_{A^1}(g) \subset E$.
Clearly $g \in I \subset A$, and $g q = q g = bq$ as before.  
 \end{proof}   

{\bf Remark.}    By considering examples such as $C_0((0,1])/C_0((0,1)) \cong \Cdb$ one sees
the necessity of the condition $0 \in E$ if $A$ is nonunital and $q \in A^{\perp \perp}$.  As in 
\cite{CSSW}, by considering the quotient of the disk algebra by an approximately unital  codimension 2 ideal,
 one sees
the necessity of the condition that $E$ not be a line segment.

\bigskip

Our best noncommutative peak interpolation result \cite[Theorem 3.4]{Bnpi} (and its variant \cite[Corollary 5.4]{BRII}) should also have `positive/Tietze versions' analogous to the two cases
considered in  Theorems \ref{peakthang1} and  \ref{peakthang1} above.
However there is an obstacle to using the approach for the latter results to improve
\cite[Theorem 3.4]{Bnpi}  say.  Namely the quotient one now has to deal with is $(If)/(Df)$ as opposed
to $I/D$.
(We remark that unfortunately  in the proof of \cite[Theorem 3.4]{Bnpi}
we forgot to repeat that $f = d^{-\frac{1}{2}}$,
as was the case in the earlier proof from that paper that it is mimicking.)    This is not
an operator algebra quotient, and so we are not sure at this point how to deal with it.   
We remark that  the Tietze variant here initially seems promising, since the key tool above
 used in that case is the numerical range lifting result from \cite{CSSW}, and this is stated 
in that paper in remarkable generality.  However we were not able to follow the proof of
the latter in this generality, although as we said at the end of \cite{BOZ} we were able to 
verify it in the less general setting needed in the last proof, and in \cite{BOZ}.

\medskip

\medskip

{\em Acknowledgements.}   As we said earlier, much of 
the present paper was formerly part of ArXiV preprint \cite{BRIII} from 2013.  The first author wishes to thank the departments at the Universities of 
Leeds and Lancaster, and in particular the second author and Garth Dales,
for their warm hospitality during a visit in April--May 2013.     We also gratefully acknowledge   support from UK research council
grant  EP/K019546/1
for largely funding that visit.  We thank Ilya Spitkovsky 
for some clarifications regarding roots of operators and results in \cite{MP}, and Alex Bearden for spotting several typos
in the manuscript.   Many of the more recent results here were found or 
presented during  the Thematic Program (on abstract harmonic analysis, Banach and operator algebras) and COSy 2014 at the Fields Institute, whom we thank for 
 the opportunity to visit and speak.

\end{document}